\documentclass[12pt]{article}

\usepackage[margin=2.5cm, top=2.5cm, bottom=2.5cm]{geometry}

\usepackage[utf8]{inputenc} 
\usepackage[T1]{fontenc}    

\usepackage{url}            
\usepackage{booktabs}       
\usepackage{amsfonts}       
\usepackage{nicefrac}       
\usepackage{microtype}      
  
\usepackage{amsmath,amssymb,amsthm,xcolor}
\usepackage{array}
\usepackage{float}

\usepackage{wrapfig}
\usepackage{thmtools}
\usepackage[hidelinks,hypertexnames=false]{hyperref}
\usepackage[capitalise]{cleveref}

\theoremstyle{plain}
\newtheorem{theorem}{Theorem}[section]
\declaretheorem[
  name=Fact,          
  sibling=theorem,      
  refname={Fact,Facts}   
]{fact}
\declaretheorem[name=Lemma,       sibling=theorem, refname={Lemma,Lemmas}]{lemma}

\declaretheorem[name=Proposition, sibling=theorem, refname={Proposition,Propositions}]{proposition}
\declaretheorem[name=Assumption,  sibling=theorem, refname={Assumption,Assumptions}]{assumption}
\declaretheorem[name=Example,     sibling=theorem, refname={Example,Examples}]{example}
\declaretheorem[name=Remark,      sibling=theorem, refname={Remark,Remarks}]{remark}
 \declaretheorem[
  name=Definition,
  sibling=theorem,
  refname={Definition,Definitions},
  style=definition
]{definition}

\usepackage{graphicx}
\usepackage{subcaption}
\usepackage{mdframed,framed}
\usepackage{lipsum}
\usepackage{tikz,calc}
\usepackage{enumitem}
\usetikzlibrary{calc}
\allowdisplaybreaks
\usepackage{pgfplots}
\usetikzlibrary{intersections, pgfplots.fillbetween}
\usepackage{epstopdf}
\usepackage{algorithmic}
\usepackage{booktabs}
\usepackage{changepage}
\usepackage{centernot}
\usepackage{stmaryrd}
\usepackage[nottoc]{tocbibind}
\newcommand{\forae}{%
  \tikz[baseline={(forall.base)}]{
    \node[inner sep=0pt, outer sep=0pt] (forall) {$\forall$};
    \fill[white] (-0.06em,0.05ex) rectangle (0.06em,0.25ex);
    \draw[line width=0.04em] (-0.06em, 0.7ex) -- (0.06em, 0.7ex);
  }
}
\usepackage{empheq}

\newcommand{\R}{\mathbb{R}}

\newcommand{\cP}{\mathcal{P}}

\newcommand{\dom}{\mathrm{dom}}

\newcommand{\gph}{\mathrm{gph}}

\newcommand{\N}{\mathbb{N}}
\newcommand{\eR}{\overline\R}

\def\rint{\mathrm{relint}}
\def\p{^}
\newcommand{\co}{\mathrm{co}\hspace*{.3mm}}
\newcommand{\cco}{\overline{\mathrm{co}}\hspace*{.3mm}}

\usepackage{accents}

\def\sign{\mathrm{sign}}
\def\sgn{\mathrm{sgn}}

\usepackage{accents}

\title{\LARGE Lyapunov stability of the Euler method}

\begin{document}

\author{\large C\'edric Josz\thanks{\url{cj2638@columbia.edu}, IEOR, Columbia University, New York.}}
\date{}

\maketitle

\begin{center}
    \textbf{Abstract}
    \end{center}
    \vspace*{-3mm}
 \begin{adjustwidth}{0.2in}{0.2in}
 ~~~~ We extend the Lyapunov stability criterion to Euler discretizations of differential inclusions. It relies on a pair of Lyapunov functions, one in continuous time and one in discrete time. 
 In the context of optimization, this yields sufficient conditions for the stability of nonisolated local minima when using the Bouligand subgradient method.
\end{adjustwidth} 
\vspace*{3mm}
\vspace*{3mm}



\tableofcontents

\section{Introduction}
\label{sec:Introduction}


Given a set-valued mapping $F:\mathbb{R}^n\rightrightarrows\R^n$, a point $\overline{x}\in \mathbb{R}^n$ is Lyapunov stable \cite{liapounoff1907probleme} if
$$ \forall \epsilon >0, ~ \exists \delta>0: ~~~ 
        x(0) \in B_{\delta}(\overline{x}) ~~~ \Longrightarrow ~~~ 
        x([0,T)) \subseteq B_{\epsilon}(\overline{x}) $$
where $x:[0,T)\rightarrow \mathbb{R}^n$ is any absolutely continuous curve with $T\in(0,\infty]$ such that 
$$ \forae t\in(0,T),\quad    x'(t) \in F(x(t)). $$
Above, $\forae$ means for almost every and $B_r(\overline{x})$ is the open ball of radius $r$ centered at $\overline{x}$ with respect to the Euclidean norm $|\cdot|=\sqrt{\langle \cdot,\cdot\rangle}$. Following \cite{josz2023lyapunov}, we adapt this notion to the Euler discretization by requiring that
\begin{equation*}
        \forall \epsilon >0, ~ \exists \delta,\overline{\alpha}>0: ~ \forall \{\alpha_k\}_{k\in \N} \subseteq (0,\overline{\alpha}], ~~~ 
        x_0 \in B_{\delta}(\overline{x}) ~~~ \Longrightarrow ~~~ 
        \{x_k\}_{k\in \N} \subseteq B_{\epsilon}(\overline{x})
\end{equation*}
where $\{x_k\}_{k\in \N}$ is any sequence in $\mathbb{R}^n$ such that 
\begin{equation}
\label{eq:euler}
    \forall k \in \N, ~~~ x_{k+1} \in x_k + \alpha_k F(x_k).
\end{equation}
Such points will be referred to as d-stable. 

The order of the quantifiers is important. 
For example, with $F:\R\rightrightarrows \R$ defined by $F(x) = -\sign(x)$ where
\begin{equation*}
    \sign(x)=\left\{\begin{array}{cl}
        x/|x| & \text{if}~ x\neq 0, \\
        \left[-1,1\right] & \text{if}~ x=0,
    \end{array}\right.
\end{equation*}
the origin is d-stable. But fixing a sequence $\alpha_0,\alpha_1,\hdots$, however small, then requiring arbitrary $\epsilon$ accuracy for some $\delta$ initialization, is impossible. 
The notion of d-stability is hence not captured by existing definitions regarding difference inclusions, in particular the rich framework of Krasovskii-Lyapunov stability, known as $\mathcal{KL}$-stability \cite[Definition 2.1]{kellett2005robustness}. Accordingly, there are no prior works on the Lyapunov stability of the Euler method for differential inclusions, to the best of our knowledge. 

In its most general form, the Lyapunov stability criterion can be conveniently stated using set-valued conservative fields \cite[Lemma 2]{bolte2020conservative}. 
Given a locally Lipschitz function $f:\R^n\to \R$, a set-valued map $D:\R^n\rightrightarrows\R^n$ is a conservative field for $f$, called potential, if $D$ has a closed graph with nonempty compact values and 
$$  \forae t\in (0,1), ~ \forall v \in D(x(t)),~~~ (f\circ x)'(t)=\langle v, x'(t) \rangle $$
for any absolutely continuous curve $x:[0,1]\to\R^n$. A scalar $\ell$ is a $D$-critical value of $f$ if there exists $x\in \R^n$ such that $\ell=f(x)$ and $0\in D(x)$. 

\begin{theorem}[{Lyapunov stability criterion}]
\label{thm:lyapunov}
 Let $\overline{x}\in\R^n$ and $F:\R^n\rightrightarrows\R^n$. Suppose:
    \begin{enumerate}[label=\rm{(\rm{\roman*})}]
        \item \label{item:negative} $f:\R^n\to\R$ is a potential of a conservative field $D:\R^n\rightrightarrows \R^n$ such that 
        $$\forall x\in\R^n,~\sup_{u\in F(x)}\min_{v\in D(x)}\langle u,v\rangle \leq 0;$$
        \item \label{item:strict} $\overline x$ is a strict local minimum of $f$.
    \end{enumerate}
Then $\overline x$ is stable. 
\end{theorem}

The crux of the proof, included in the Appendix, is that $f\circ x$ is decreasing for any trajectory $x:[0,T)\to\R\p n$ of $F$: 
$$\forall s,t \in [0,T),~~~ s\leq t ~~~ \Longrightarrow ~~~  f(x(t)) \geq f(x(s)).$$
We will refer to any function $f:\R^n\to\R$ satisfying this property as a Lyapunov function. Accordingly, we say that $g:\R^n\to\eR$ where $\eR=\R\cup\{\infty\}$ is a d-Lyapunov function on $U \subseteq \R^n$ if
\begin{equation*}
        \exists \overline{\alpha}>0: ~ \forall \{\alpha_k\}_{k\in \N} \subseteq (0,\overline{\alpha}], ~ \forall k \in \N, ~~~ x_k  \in U  ~\Longrightarrow~
        g(x_{k+1})\leq  g(x_k)
\end{equation*}
where $\{x_k\}_{k\in \N}$ is any sequence satisfying \eqref{eq:euler}. 

Let $\cP$ be the set of continuous functions $\kappa:\R\p n\to \R$ that are positive definite, i.e., such that $\kappa(0)=0$ and $\kappa(x)>0$ for all $x\in \R\p n\setminus\{0\}$. Suppose that, in addition to \ref{item:strict}, $F:\R\p n\rightrightarrows\R\p n$ is upper semicontinuous with nonempty compact convex values and \ref{item:negative} is strengthened to:
\begin{enumerate}[label=\rm{(\rm{\roman*}$'$)},start=1]
        \item \label{item:negative_definite} $f:\R^n\to\R$ is a potential of a conservative field $D:\R^n\rightrightarrows \R^n$ such that 
        $$\forall \widehat{x}\in \R\p n,~ \exists (\rho,\kappa)\in (0,\infty)\times \cP,~ \forall x\in B_{\rho}(\widehat{x}),~\sup_{u\in  F(x)} \min_{v\in D(x)} \langle u,v\rangle \leq -\min_{w\in D(x)}\kappa(w)$$
        and $f(\overline x)$ is an isolated $D$-critical value of $f$.
\end{enumerate}
Then $\overline x$ is asymptotically stable. This means that it is stable and all maximal trajectories are globally defined and converge to $\overline x$ when initialized nearby \cite[Definition 2.1]{clarke1998asymptotic}. 

In this paper, we seek to extend the Lyapunov criterion to d-stability. This is motivated by the desire to better understand the training dynamics of deep neural networks. The algorithms used can often be viewed as discretizations of dynamical systems. They are believed to stabilize near flat minima of the training loss \cite{hochreiter1997flat,keskar2017large,wu2018sgd}, which may have good generalization properties \cite{hochreiter1997flat,andriushchenko2023modern}. Stability has been of recent interest for these reasons \cite{josz2023globalstability,xiao2023sgd,bolte2025inexact}. 
Our first main result is as follows. 

\begin{theorem}
\label{thm:stable_point}
Let $\overline x \in \R^n$. Suppose:
    \begin{enumerate}[label=\rm{(\rm{\roman*})}]
        \item $F:\R^n\rightrightarrows \R^n$ 
        is upper semicontinuous with nonempty compact values;
        \item $f:\R^n\to\R$ is a potential of a conservative field $D:\R^n\rightrightarrows\R^n$ such that 
        $$\forall \widehat{x}\in \R\p n,~ \exists (\rho,\kappa)\in (0,\infty)\times \cP,~ \forall x\in B_{\rho}(\widehat{x}),~\max_{u\in \co F(x)} \min_{v\in D(x)} \langle u,v\rangle \leq -\min_{w\in D(x)}\kappa(w)$$
         and $f(\overline x)$ is an isolated $D$-critical value of $f$;
        \item $g:\R^n\to \eR$ is continuous and d-Lyapunov near $\overline x$;
        \item $\overline x$ is a strict local minimum of $f+g$.
    \end{enumerate}
Then $\overline x$ is d-stable.
\end{theorem}

While we are not aware of any prior results on d-stability, \cite[Theorem 2]{josz2023lyapunov} proves that strict local minima of locally Lipschitz semi-algebraic functions $f:\R\p n \to \R$ are d-stable if one restricts the definition to constant step sizes and $F$ is the Clarke subdifferential $-\overline \partial f$. \cref{thm:stable_point} improves on this result in several ways:
\begin{enumerate}[itemsep=1mm] 
    \item it allows for variable step sizes; \label{step_size}
    \item it drops the semi-algebraic requirement using a new argument that no longer relies on the Kurdyka-\L{}ojasiewicz inequality \cite{kurdyka1998gradients}; \label{semi-algebraic}
    \item it applies to set-valued mappings $F$ that need not be conservative fields; \label{not_conservative}
    \item it relaxes the requirement that $\overline{x}$ be a strict local minimum of $f$ to being a strict local minimum of $f+g$ for some d-Lyapunov function $g$. \label{d-Lyapunov}
\end{enumerate}
Items \ref{step_size} unlocks the possibility of extending \cref{thm:stable_point} to asymptotic d-stability, as we will do. Item \ref{semi-algebraic} shows that semi-algebraic geometry plays no role in the theory of d-stability, contrary to what our previous work suggests.
Item \ref{not_conservative} enables the application to normalized gradient descent, particularly suitable for smooth Lyapunov functions $f$, and studied in the sequel \cite{josz2026implicit}. Item \ref{d-Lyapunov} allows one to prove d-stability of nonstrict local minima of $f$. This is a key novelty of this work, which calls for an entirely new proof strategy.

The proof of \cref{thm:stable_point} harnesses the synergy between the Lyapunov function $f$ and the d-Lyapunov $g$.
They play complementary roles: $f(x_k)$ remains bounded while $g(x_k)$ decreases, so that $f(x_k)+g(x_k)$ is bounded, despite not being monotone. Since $\overline{x}$ is a strict local minimum of $f+g$, the sequence $x_k$ is itself bounded.
In fact, it suffices if $g(x_k)$ decreases when it is bounded away from $g(\overline{x})$, from above. This allows one to relax the d-Lyapunov requirement near $\overline{x}$, paving the way for future extensions. 

In order to obtain convergence, it is natural to ask for a stronger decrease property 
$$g(x_{k+1})-g(x_k)\leq -\omega \alpha_k\p p$$
for some $\omega>0$ and $p\in[1,\infty)$, when $g(x_k)$ does decrease. If $$\sum_{k=0}\p \infty \alpha_k\p p = \infty ~~~\text{and}~~~ \alpha_k\to 0,$$ then one can show that $f(x_k)\to f(\overline{x})$ and $g(x_k)\to g(\overline{x})$, despite neither being monotone, yielding $x_k\to\overline{x}$. The absence of monotonicity is a key challenge in the proof, with $f$ and $g$ mutually helping each other. The boundedness of $g(x_k)$ and the nonsummability of $\alpha_k$ yield $f(x_k)\to f(\overline{x})$. The convergence of $f(x_k)$ to $f(\overline{x})$ 
and the nonsummability of $\alpha_k\p p$ in turn imply that $g(x_k)\to g(\overline{x})$. 

Caution though: without nonsummability of $\alpha_k\p p$, the iterates $x_k$ could converge prematurely to a point $\widetilde x$ with $g(\widetilde x)>g(\overline{x})$. The idea of using sufficiently slowly diminishing step sizes to converge to d-stable points is a key contribution of this work. A prime choice is of course
$$\forall k \in \N, ~~~ \alpha_k = \frac{\beta}{(k+1)\p {1/\gamma}},$$
with $\gamma\in [p,\infty)$ for some sufficiently small $\beta>0$. This induces a trade-off in practice. Low values of $\gamma$ favor reduction of $f(x_k)$, while high values of $\gamma$ favor reduction of $g(x_k)$. Taking `$\gamma=\infty$' amounts to using constant step size, ideal for $g(x_k)$, but now $f(x_k)$ is merely bounded. This is a fine choice if one is only interested in stability.

Our study is orthogonal to existing theory on the asymptotic stability of differential inclusions. The converse Lyapunov theorem of Clarke, Ledyaev, and Stern \cite[Theorem 1.2]{clarke1998asymptotic}, which applies to upper semicontinuous mappings $F:\R\p n \rightrightarrows \R\p n$ with nonempty convex compact values, shows that asymptotic stability is equivalent to the existence of a smooth strong Lyapunov function \cite[Definiton 1.1]{clarke1998asymptotic}. The proof of necessity shows that asymptotic stability is robust to certain perturbations on $F$ \cite[Proposition 3.1]{clarke1998asymptotic}, then builds a smooth strong Lyapunov function for the perturbed system, which is a fortiori a smooth strong Lyapunov function the original system.

While the Euler method \cite{dontchev1989error,donchev1998stability,sandberg2009convergence} can be viewed as a solution to a perturbed differential inclusion \cite[Equation (13)]{dontchev1992difference} when $F$ is Lipschitz continuous \cite[Definition 9.26]{rockafellar2009variational}, it does not necessarily satisfy the perturbed system of Clarke \textit{et al}, and certainly not when $F$ fails to be Lipschitz continuous. For example, with $F(x) = -\sign(x)\sqrt{|x|}$, the stability of the origin is not preserved upon discretizing with constant step size.
The situation is worse still when $F$ is not continuous \cite[Definition 5.4]{rockafellar2009variational}, where stability is not preserved for any choice of variable step sizes (as the first example $F(x)=-\sign(x)$ demonstrates).

Clarke \textit{et al.}'s breakthrough \cite{clarke1998asymptotic}, which generalizes the converse Lyapunov theorems of Massera \cite{massera1949liapounoff} and Kurzweil \cite{kurzweil1955reversibility} for differential equations, was later extended to $\mathcal{KL}$-stability \cite[Theorem 1]{teel2000smooth} and instability \cite[Theorem 3.15]{braun2021stability} \cite{braun2018complete}. A powerful result on the asymptotic stability of the Euler method for differential equations, the Kloeden-Lorentz theorem \cite{kloeden1986stable,han2017attractors}, relies on the converse Lyapunov theorem of Yoshizawa \cite{yoshizawa1966stability}. However, there seem to be no results dealing with mere stability of the Euler method, as opposed to asymptotic stability, even for differential equations \cite{lakshmikantham2002theory,grune2016nonlinear}.


The paper is organized as follows. \cref{sec:background} contains background material on differential inclusions. \cref{sec:Theory} develops the theory of d-stability. Necessary and sufficient conditions for being a d-Lyapunov function are discussed in \cref{subsec:Lyapunov functions}. \cref{subsec:Stable points} contains the proof of \cref{thm:stable_point}, which serves as a template for ensuing results. \cref{subsec:Stable points} extends the notion of d-stability from points to sets, and to asymptotic d-stability. \cref{subsec:attractors} studies a subclass of d-stable sets, called attractors. \cref{sec:Application} applies theory of d-stability to the Bouligand subdifferential.

\section{Background}
\label{sec:background}

We first introduce some notations. As usual, $\N=\{0,1,2,\hdots\}$, $\N\p * = \N\setminus \{0\}$, $\R\p * = \R\setminus\{0\}$, and $\R_+=[0,\infty)$. Given two integers $a \leq b$, let $\llbracket a,b \rrbracket = \{a,a+1,\hdots,b\}$. The ceiling $\lceil t \rceil$ of a real number $t$ is the unique integer such that $\lceil t \rceil-1 < t \leq \lceil t \rceil$. The symbol $\land$ means `and', and $\lor$ means `or'. Given a set $A \subseteq \R\p n$, $A\p c = \R\p n \setminus A$ denotes the complement of $A$. If $L:U\to V$ is a linear map between two finite dimensional inner product spaces, then $L\p *:V\to U$ denotes the adjoint. Let $\|\cdot\|_F$ and $\|\cdot\|_1$ respectively denote the Frobenius norm and entrywise $\ell_1$-norm. Also, $\langle \cdot,\cdot\rangle_F$ denotes the Frobenius norm.
Let $\overline{(\cdot)}$, $\mathrm{int}$, $\mathrm{relint}$, and $\mathrm{co}$ respectively denote the closure, interior, relative interior, and convex hull.
Given a point $x \in \mathbb{R}^n$, we use 
\begin{equation*}
    d(x,X) = \inf \{|x-y|: y \in X\}~~~\text{and}~~~
    P_X(x) = \mathrm{argmin} \{ |x-y| : y \in X\}
\end{equation*}
to denote the distance to and the projection on a subset $X\subseteq \mathbb{R}^n$, respectively. 
Given $X\subseteq \R\p n$ and $r\geq 0$, we extend the notion of ball as follows:
$$B_r(X) = \{ x\in \R\p n : d(x,X)< r\}.$$
Given $f:\R^n\to\eR$ and $\ell \in {\mathbb{R}}$, let 
$$[f = \ell] = \{ x \in \mathbb{R}^n : f(x) = \ell \}$$
and define $[f \leq \ell]$ similarly. In particular, $\arg\min f = [f=\min f]$. The domain and the graph of $f$ are defined by 
$$\dom f = \{ x \in \R\p n : f(x)<\infty\} ~~~\text{and}~~~\gph f = \{ (x,\ell) \in \R\p {n+1} : f(x)=\ell\}.$$
For a set-valued mapping $F:\mathbb{R}^n \rightrightarrows \mathbb{R}^m$, the domain and the graph are respectively defined by
$$
      \dom F=\{x\in \R^n:~F(x)\neq \emptyset \} ~~~\text{and}~~~
      \gph F = \{ (x,y)\in \mathbb{R}^n\times \R^m : F(x)\ni y \}.
$$
The image of $X\subseteq \R\p n$ and the preimage of $y \in \R\p m$ are given by 
$$F(X) = \bigcup_{x \in X} F(x) ~~~\text{and}~~~ F^{-1}(y) = \{ x\in \mathbb{R}^n : F(x) \ni y \}.$$
The convexified mapping $\co F:\R\p n\to \R\p n$ is defined by $(\co F)(x) = \co F(x)$. 

A subset $A$ of $\mathbb{R}^n$ is semi-algebraic \cite{bochnak2013real} if it is a finite union of basic semi-algebraic sets, which are of the form 
    \begin{equation*}
        \{ x \in \mathbb{R}^n : f_1(x) > 0 , \hdots , f_p(x) > 0, f_{p+1}(x) = 0, \hdots , f_q(x) = 0\}
    \end{equation*}
    where $f_1,\hdots,f_q$ are polynomials with real coefficients. A function $f:\R\p n\to \R$ is semi-algebraic if $\gph f$ is semi-algebraic.

The following notion \cite[Definition 1 p. 41]{aubin1984differential} plays a key role in differential inclusions.

\begin{definition}
    A mapping $F:\R\p n \rightrightarrows \R\p m$ is upper semicontinuous at $\overline{x} \in \R\p n$ if for any neighborhood $V$ of $F(\overline{x})$, there exists a neighborhood $U$ of $\overline{x}$ such that $F(U)\subseteq V$.
\end{definition}

A related notion is as follows \cite[Definition 5.14]{rockafellar2009variational}.

\begin{definition}
    A mapping $F:\R\p n \rightrightarrows \R\p m$ is locally bounded at $\overline{x}\in \R\p n$ if there exists a neighborhood $U$ of $\overline{x}$ in $\R\p n$ such that $F(U)$ is bounded.
\end{definition}

Both notions hold without referring to a point if they hold at every point in $\R\p n$. We include a proof of the following simple facts in the Appendix for the reader's convenience.

\begin{fact}
    \label{fact:bounded}
    If $F:\R\p n \rightrightarrows \R\p n$ is locally bounded and $X\subseteq \R\p n$ is bounded, then $F(X)$ is bounded.
\end{fact}

The above definitions are related as follows.

\begin{fact}
\label{fact:usc_bounded}
Given $F:\R\p n\rightrightarrows \R^m$, the following are equivalent: 
    \begin{enumerate}[label=\rm{(\rm{\roman*})}]
        \item $F$ is upper semicontinuous with compact values;
        \item $F$ is locally bounded and has a closed graph.
    \end{enumerate}
\end{fact}

\begin{definition}
    A trajectory of $F:\R\p n \rightrightarrows\R\p n$ is an absolutely continuous curve $x:I\to\R\p n$ where $I$ is an interval of $\R$ such that
$$\forae t\in I, ~~~ x'(t) \in F(x(t)).$$
\end{definition}

We refer to solutions to the differential inclusion $\dot{x}\in F(x)$ as trajectories of $F$ over an interval $I = [0,T)$ for some $T\in(0,\infty]$ or $I = [0,T]$ for some $T\in(0,\infty)$. A solution $x:I\to\R\p n$ is maximal if for any other solution $y:J\to\R\p n$ such that $I\subseteq J$ and $x(t)=y(t)$ for all $t\in I$, we have $I=J$. A solution is globally defined if $I = [0,\infty)$. If $F:\R\p n \rightrightarrows\R\p n$ is upper semicontinuous with nonempty compact convex values, then the initial value problem 
$$\left\{
\begin{array}{cc}
     \dot{x} \in F(x),  \\
     x(0) = x_0,
\end{array}
\right.$$
admits a solution for any initial condition $x_0\in \R\p n$ by \cite[Theorem 3 p. 98]{aubin1984differential}. As a result, bounded maximal solutions are globally defined. For other existence results, including global existence, see Davy \cite[Theorem 4.2]{davy1972properties} and Seah \cite[Theorem 3.3]{seah1982existence}. 
Similar to differential equations, stable points must be critical points.

\begin{fact}
    \label{fact:equilibrium}
    Let $F:\R\p n\rightrightarrows \R^n$ be upper semicontinuous with nonempty compact convex values. If $\overline{x}\in\R\p n$ is stable, then $0\in F(\overline{x})$.
\end{fact}

Given the importance of convex values for local existence, sometimes it is useful to convexify a given mapping. The following fact comes in handy.

\begin{fact}
    \label{fact:neighborhood}
    Given a compact set $X\subseteq \R^n$ and a neighborhood $U$ of $X$, there exists $r>0$ such that $B_r(X)\subseteq U$.
\end{fact}

\cref{fact:neighborhood} can be used to show that convexifying preserves good properties.

\begin{fact}
    \label{fact:co_usc}
    If $F:\R^n\rightrightarrows\R^n$ is upper semicontinuous with compact values, then $\co F$ is upper semicontinuous with convex compact values.
\end{fact}

\section{Theory}
\label{sec:Theory}

This section develops the theory of d-stability.

\subsection{Lyapunov functions}
\label{subsec:Lyapunov functions}
We begin with two necessary conditions for a function to be d-Lyapunov, given an arbitrary mapping $F:\R\p n \rightrightarrows\R\p n$.

\begin{proposition}
\label{prop:dL_necessary}
    If $g:\R^n\to\eR$ is d-Lyapunov and Lipschitz continuous near $\overline{x}$, then $g$ is Lyapunov near $\overline{x}$. 
\end{proposition}
\begin{proof}
Let $x:[0,T]\to \R^n$ be a trajectory of $F$ taking values near $\overline{x}$. Since $x$ is absolutely continuous and $g$ is Lipschitz continuous near $\overline{x}$, the composition $g\circ x$ is absolutely continuous. Thus $x$ and $g\circ x$ are almost surely differentiable. For almost every $t\in (0,T)$, we have 
\begin{align*}
    (g\circ x)'(t)&=\lim_{h\to 0}\frac{g(x(t+h))-g(x(t))}{h}\\
    &=\lim_{h\to 0}\frac{g(x(t)+hx'(t)+o(h))-g(x(t))}{h}\\
    &=\lim_{h\to 0}\frac{g(x(t)+hx'(t))+o(h)-g(x(t))}{h}\leq 0.
\end{align*}
Since $g\circ x$ is absolutely continuous, for all $0\leq a\leq b\leq T$, 
$  (g\circ x)(b)-(g\circ x)(a)=\int_a^b (g\circ x)'(t)dt\leq 0.     $
\end{proof} 
\begin{proposition}
    \label{prop:dL_necessary_D1}
    If $g:\R^n\to\eR$ is d-Lyapunov and differentiable at $x$, then for all $u \in F(x)$, $\langle \nabla g(x),u\rangle\leq 0$. In particular, if $0 \in \rint F(x)$, then $\langle \nabla g(x) , F(x) \rangle=\{0\}$.
\end{proposition}
\begin{proof}
    There exists $\overline{\alpha}>0$ such that 
$$ \forall \alpha\in (0,\overline{\alpha}],~\forall u\in F(x),\quad 0\leq g(x+\alpha u)-g(x)=\alpha\langle \nabla g(x),u \rangle+ o(\alpha)$$
and thus $\langle \nabla g(x),u \rangle\leq 0$ for all $u\in F(x)$. If $0 \in \rint F(x)$, then for all $u \in F(x)$, $\lambda u \in F(x)$ for all sufficiently small $\lambda \in \R$. Thus $0 \leq \langle \nabla g(x), \lambda u \rangle = \lambda \langle \nabla g(x),u \rangle$ and in particular $\langle \nabla g(x),u \rangle\geq 0$. 
\end{proof}

\begin{proposition}
    \label{prop:sufficient_dL1}
    Let $F:\R\p n \rightrightarrows \R\p n$ be locally bounded with a closed graph and $g:\R\p n \to \eR$ be $C\p {1,1}$ near $\overline{x}\in \R\p n$. If $\sup_{u\in F(\overline{x})} \langle \nabla g(\overline{x}) , u \rangle<0$, then
    $$\exists \overline{\alpha},\rho,\omega
    >0:~ \forall \alpha \in (0, \overline{\alpha}], ~ \forall x\in B_{\rho}(\overline{x}),~\forall u\in F(x),~~ g(x+\alpha u)-g(x) \leq -\omega\alpha.$$
    \end{proposition}
\begin{proof}
Let $\omega = -\sup_{u\in F(\overline{x})} \langle \nabla g(\overline{x}) , u \rangle/3$.
Since $F$ is locally bounded with a closed graph and $g$ is $C\p 1$ near $\overline{x}$, there exists $\rho>0$ such that 
$$\forall x\in B_{2\rho}(\overline{x}),~ \forall u \in F(x), ~~ \langle \nabla g(x),u\rangle < -2\omega.$$
Otherwise, there exists $x_k \to \overline{x}$ and $u_k \in F(x_k)$ such that $\langle \nabla g(x_k),u_k\rangle \geq -2\omega$. Since $F$ is locally bounded, $u_k$ is bounded and so $u_k\to \overline{u}$ up to a subsequence. As $F$ has a closed graph, $\overline{u}\in F(\overline{x})$. By continuity of $\nabla g$, one may pass to the limit $\langle \nabla g(\overline{x}),\overline{u}\rangle \geq -2\omega > -3\omega = \sup_{u\in F(\overline{x})} \langle \nabla g(\overline{x}) , u \rangle$, a contradiction. 

Since $g$ is $C\p {1,1}$ on $B_{2\rho}(\overline{x})$ after possibly reducing $\rho$, by \cite[Lemma 1.2.4]{nesterov2003introductory} there exists $L>0$ such that
$$\forall x,y \in B_{2\rho}(\overline{x}), ~~~ \left|g(y) - g(x) - \langle \nabla g(x), y-x \rangle\right| \leq L|x-y|\p 2/2.$$
As $F$ is locally bounded, there exists $R>0$ such that $|u| \leq R$ for all $x \in B_{\rho}(\overline{x})$ and $u\in F(x)$. Let $\overline{\alpha} = \min\{ 2\omega/(LR\p 2) , \rho/R\}$. For all $x \in B_{\rho}(\overline{x})$, $u\in F(x)$, and $\alpha \in (0, \overline{\alpha}]$, we have
  \begin{equation*}
     g(x+\alpha u) - g(x) \leq \langle \nabla g(x) , u \rangle \alpha + L|u|\p2 \alpha \p 2 /2 \leq -( 4\omega - L \alpha|u|\p 2 ) \alpha /2 \leq - \omega \alpha.\qedhere
  \end{equation*}
\end{proof}

The following proposition is proved similarly.

\begin{proposition}
    \label{prop:sufficient_dL2}
    Let $F:\R\p n \rightrightarrows \R\p n$ be locally bounded with a closed graph and $g:\R\p n \to \eR$ be $C\p {2,2}$ near $\overline{x}\in \R\p n$. If 
    $$\exists r>0:~ \forall x \in B_r(\overline{x}),~ \forall u\in F(x),~ \langle \nabla g(x) , u \rangle \leq 0 ~~~\land ~~~ \sup_{u\in F(\overline{x})} \langle \nabla\p2 g(\overline{x})  u , u \rangle<0,$$
    then
    $~~~\exists \overline{\alpha},\rho
    >0:~ \forall \alpha \in (0, \overline{\alpha}],~ \forall x\in B_{\rho}(\overline{x}),~\forall u\in F(x),~~  g(x+\alpha u)-g(x) \leq -\omega\alpha\p 2$.
    \end{proposition}

These sufficient conditions naturally lead us to single out a class of d-Lyapunov functions.

\begin{definition}
    \label{def:pdL}
    Given $F:\R\p n \rightrightarrows \R\p n$, $g:\R^n\to\eR$ is $p$-d-Lyapunov on $U \subseteq \R^n$ with $p \in [1,\infty)$ if
    $$\exists \overline{\alpha},\omega>0: ~ \forall \alpha\in (0,\overline{\alpha}], ~ \forall x \in U,~ \forall u\in F(x), ~~ g(x+\alpha u) - g(x) \leq - \omega \alpha\p p.$$
\end{definition}

We choose to write the theory with linear decrease but we could very well write it with a multiplicative decrease $g(x+\alpha u) \leq e\p {- \omega \alpha\p p}g(x)$ with nonnegative $g$. Linear decrease technically captures multiplicative decrease by exponentiating, but this requires allowing $g$ to take the value $-\infty$ at stable points, which we prefer to avoid.
\cref{def:pdL} admits a multistep extension, which will be useful for studying normalized gradient descent \cite{josz2026implicit}.

\begin{definition}
    \label{def:pqdL}
    Given $F:\R\p n \rightrightarrows \R\p n$, $g:\R^n\to\eR$ is $(p,q)$-d-Lyapunov on $U \subseteq \R^n$ with $p\in[1,\infty)$ and $q \in \N\p *$ if there exist $\overline{\alpha},\omega>0$ such that 
\begin{equation*}
        \forall \{\alpha_k\}_{k\in\N} \subseteq (0,\overline{\alpha}], ~ \forall k \in \N,~~~ x_k \in U ~~~ \Longrightarrow ~~~ g(x_{k+q}) - g(x_k) \leq - \omega \min\{\alpha_k,\hdots,\alpha_{k+q-1}\}\p p
\end{equation*}
for any $\{x_k\}_{k\in\N}\subseteq \R\p n$ such that $x_{k+1}\in x_k +\alpha_k F(x_k)$ for all $k\in\N$.
\end{definition}

Notice that $(p,1)$-d-Lyapunov and $p$-d-Lyapunov mean the same thing. If $F:\R\p n\rightrightarrows\R\p n$ is locally bounded and $g:\R\p n \to \eR$ is $p$-d-Lyapunov on $B_{\rho}(\overline{x})$ for some $\rho>0$ and $\overline{x}\in\R\p n$, then for any $q \in \N\p *$, there exists $\rho'\in (0,\rho)$ such that $g$ is $(p,q)$-d-Lyapunov on $B_{\rho'}(\overline{x})$. On the other hand, a $(p,q)$-d-Lyapunov function need not be a d-Lyapunov function.

\subsection{Stable points}
\label{subsec:Stable points}

As with Lyapunov functions, it is natural to wonder whether stable and d-stable points are related. As \cref{eg:ellipse} will show, stable points can be d-unstable. The converse is also possible.

\begin{example}
    \label{eg:d-stable}
    Let 
    $$ F(x,y) = 
    -\begin{pmatrix}
    x \\ 2y
    \end{pmatrix} ~~\text{if}~~ x\p 2 \neq y
    ~~~~\text{and} ~~~~
    F(x,y) = 
    \co\left\{-\begin{pmatrix}
    x \\ 2y
    \end{pmatrix},
    \begin{pmatrix}
    x \\ 2y
    \end{pmatrix}\right\}
    ~~\text{if}~~ x\p 2 = y.
    $$
    The origin is unstable but d-stable.
\end{example}
\begin{proof}
    The trajectory $(x(t),y(t))=(x_0e\p t,y_0e\p{2t})$ initialized at $(x_0,y_0)\in\R\p 2$ with $x_0\p 2=y_0\neq 0$ diverges. With $\overline{\alpha} = 1/2$, $x\p 2 = y \implies [(x\p + = x \land y\p + = y) \lor ((x\p +)\p 2 > y\p +)]$ and $x\p 2>y \implies (x\p +)\p 2 > y\p +$. Indeed, there exists $\lambda \in [-1,1]$ such that $(x\p +)\p 2-y\p + = (x-\alpha \lambda x)\p 2 - (y - 2\alpha \lambda y) = (1-2\alpha \lambda)(x\p 2-y)+\alpha\p 2 \lambda \p 2 x\p 2$. It is then easy to see that discrete trajectories are eventually constant at a point $(x,y)$ where $x\p 2=y$ and $|(x,y)|\leq |(x_0,y_0)|$, or they converge to $(0,0)$.
\end{proof}

Essential to the proof of d-stability is the upcoming tracking result. It is a simple generalization of \cite[Lemma 1]{josz2023lyapunov} to variable step sizes. 
The proof is included in the Appendix for completeness. Prior results on the distance between solutions to differential inclusions and their Euler discretizations assume Lipschitz continuity and linear growth with convex values \cite[Theorem]{dontchev1989error} \cite[Theorem 2.4]{dontchev1992difference} \cite[Theorem 4.3]{donchev1998stability}, 
or Lipschitz continuity and uniformly bounded nonconvex values \cite[Theorem 2.5]{sandberg2009convergence}.
We substitute the Lipschitz, growth, and boundedness assumptions with a uniform bound on the sequences generated by the Euler method. Also, our conclusion allows for uniformity with respect to the initial point and the step sizes.
Here and below, let $\{t_k\}_{k\in \N}$ be such that 
$$t_0 = 0 ~~~\text{and} ~~~ \forall k\in \N\p*,~~ t_k = \alpha_0+\cdots+\alpha_{k-1}$$ 
whenever a sequence $\{\alpha_k\}_{k\in\N}$ is given.
\begin{lemma}
\label{lemma:tracking}
    Let $F:\mathbb{R}^n\rightrightarrows\mathbb{R}^n$ be upper semicontinuous with nonempty compact values. Let $X\subseteq \mathbb{R}^n$ be bounded and $T>0$. Assume there exist $r,\widehat\alpha>0$ such that for all $\{\alpha_k\}_{k\in \N} \subseteq (0,\widehat\alpha]$ and for all $\{x_k\}_{k\in\N} \subseteq \mathbb{R}^n$ such that
    \begin{equation}
    \label{eq:Euler}
                \forall k \in \N,~~~ x_{k+1} \in x_k + \alpha_k F(x_k), ~~~ x_0 \in X,
    \end{equation}
    $$\text{we have}~~~~~~~~~~~~~~~~~~~~~~~~~~~~~~ \forall k \in \N, ~~~ t_k \leq T ~\Longrightarrow~  x_k\in B_r(X).~~~~~~~~~~~~~~~~~~~~~~~~~~~~~~~~~~~~~~~~$$
    Then for all $\eta>0$, there exists $\overline{\alpha}>0$ such that for all $\{\alpha_k\}_{k\in\N} \subseteq (0,\overline{\alpha}]$ and $\{x_k\}_{k\in\N} \subseteq \mathbb{R}^n$ satisfying \eqref{eq:Euler}, there exists an absolutely continuous function $x:[0,T]\rightarrow\mathbb{R}^n$ such that
    \begin{subequations}
    \begin{gather}
    \label{subeq:DI}
    \forae t \in (0,T), ~~~ x'(t) \in \co F(x(t)), ~~~ x(0) \in \overline{X},\\[2mm]
    \label{subeq:tracking}
     \forall k \in \N, ~~~ t_k \leq T ~\Longrightarrow~  |x_k-x(t_k)|\leq \eta.
     \end{gather}
    \end{subequations}
\end{lemma}

In order to convert the tracking lemma into a descent lemma,
the following assumption will come in handy.

\begin{assumption}
    \label{assume:Ff}
    Let
    \begin{enumerate}[label=\rm{(\rm{\roman*})}]
        \item $F:\R^n\rightrightarrows \R^n$ be upper semicontinuous with nonempty compact values; \label{item:F} 
        \item $f:\R^n\to\R$ be a potential of a conservative field $D:\R^n\rightrightarrows\R^n$ such that 
        $$\forall \widehat{x}\in \R\p n,~ \exists (\rho,\kappa)\in (0,\infty)\times \cP,~ \forall x\in B_{\rho}(\widehat{x}),~\max_{u\in  \co F(x)} \min_{v\in D(x)} \langle u,v\rangle \leq -\min_{w\in D(x)}\kappa(w)$$
        and $(0,\overline{\ell})\cap f(D\p {-1}(0)) = \emptyset$ for some $\overline{\ell}>0$. \label{item:f} 
\end{enumerate}
\end{assumption}

The following descent lemma revisits the arguments developed for proving d-stability with constant step sizes of strict local minima of locally Lipschitz semi-algebraic functions \cite[Theorem 2]{josz2023lyapunov}. In addition to incorporating variable step sizes, it provides a uniform upper bound on the potential function, as well as a decrease over time. These new properties will be crucial for harnessing the synergy between Lyapunov and d-Lyapunov functions.

\begin{lemma}
\label{lemma:descent}
Under \cref{assume:Ff}, let $Y\subseteq \R\p n$ be bounded and $r>0$. There exist $T>0$ and $\kappa\in \cP$ such that for all $\ell \in (0,\overline{\ell})$ and any set $X \subseteq [f\leq\ell]\cap Y$, there exists $\overline{\alpha}>0$ such that, for all $\{\alpha_k\}_{k\in\N}\subseteq (0,\overline{\alpha}]$ and $\{x_k\}_{k\in\N}\subseteq \R\p n$ satisfying \eqref{eq:Euler}, we have 
$$
\begin{array}{ccc}
   t_k \leq T & \Longrightarrow & x_k \in [f \leq 3\ell/2] \cap B_r(X), \\[2mm] 
   T/2\leq t_k \leq T & \Longrightarrow & f(x_k) \leq \ell - \min\{\ell,\zeta T\}/6,
\end{array}
$$
for all $k \in \N$ where $\zeta = \inf \{ \kappa(w): w\in D(B_r(X) \cap [\ell/2 \leq f \leq \ell]) \}$.
\end{lemma}
\begin{proof}
    Since $Y$ is bounded, \cref{assume:Ff} \ref{item:f} implies the existence of $\kappa\in\cP$ such that
    \begin{equation*}
        \forall x\in B_r(Y),~~~\max_{u\in \co F(x)} \min_{v\in D(x)} \langle u,v\rangle \leq -\min_{w\in D(x)}\kappa(w). 
    \end{equation*}
    Let $L > \sup \{|v| : v \in F(B_r(Y))\}$ be a Lipschitz constant of $f$ on $B_r(Y)$ and let $T < r/L$. Note that $L<\infty$ since $F$ is upper semicontinuous with compact values by \cref{assum} \ref{item:F}, and thus preserves boundedness by \cref{fact:bounded} and \cref{fact:usc_bounded}. Let $\{\alpha_k\}_{k\in\N}\subseteq (0,\infty)$ and $\{x_k\}_{k\in\N}\subseteq \R\p n$ be such that $x_{k+1} = x_k + \alpha_k v_k$ for some $v_k \in F(x_k)$ for all $k\in \N$ with $x_0 \in X$. For all $k\in \N$, if $t_k \leq T$, then $x_k \in B_r(X)$. Indeed, $x_0 \in B_r(X)$ and by induction
$$|x_k-x_0| = \sum_{i=0}\p {k-1} |x_{i+1}-x_i| \leq \sum_{i=0}\p {k-1} \alpha_i|v_i| \leq t_k L \leq TL < r.$$

The constant $\zeta$ is positive since $[\ell/2,\ell]\subseteq (0,\overline\ell)$ is devoid of $D$-critical values of $f$ by \cref{assume:Ff} \ref{item:f}. If it were zero, one would be reach a contradiction since $\kappa \in \cP$ and $D$ has a closed graph. Let $\eta = \min\{\ell,\zeta T\}/(3L)$. By \cref{fact:co_usc} and \cref{lemma:tracking}, there exists $\overline{\alpha}>0$ such that for all $\{\alpha_k\}_{k\in\N} \subseteq (0,\overline{\alpha}]$ and $\{x_k\}_{k\in\N} \subseteq \mathbb{R}^n$ satisfying \eqref{eq:Euler}, there exists an absolutely continuous function $x:[0,T]\rightarrow\mathbb{R}^n$ satisfying \eqref{subeq:DI}-\eqref{subeq:tracking}. Similar to the discrete case, if $t = \inf \{ s \geq 0 : |x(s)-x(0)|\geq r\}\leq T$, then one obtains the contradiction
\begin{equation*}
    r \leq |x(t)-x(0)| = \left| \int_0^t x'(s)ds \right| \leq \int_0^t |x'(s)|ds \leq tL \leq TL < r.
\end{equation*}
Thus $x([0,T]) \subseteq B_r(X)$. Let $k\in \N$ be such that $t_k \leq T$. We have
\begin{subequations}
    \begin{align}
    f(x_k) & = f(x_k) - f(x(t_k)) + f(x(t_k)) \label{level_a} \\
	   & \leq L|x_k - x(t_k)| + \max\{\ell/2,\ell - \zeta t_k\} \label{level_b} \\
          & \leq L\eta + \ell - \min\{\ell/2,\zeta t_k\} \label{level_c} \\
          & \leq \ell/2 + \ell. \label{level_d}
    \end{align}
\end{subequations}
Indeed, the bound on first term in \eqref{level_a} is the result of applying local Lipschitz continuity of $f$ to $x_k,x(t_k)\in B_r(X)\subseteq B_r(Y)$. As for the second term in \eqref{level_a}, if it is greater than or equal to $\ell/2$, then for all $t \in [0,t_k]$, we have $\ell/2 \leq  f(x(t_k)) \leq  f(x(t)) \leq f(x(0)) \leq \ell$. This holds because $D$ is a conservative field for $f$, and so $$\forae t\in(0,T),~~~ (f \circ x)'(t) = \inf_{v\in D(x(t))}\langle x'(t),v \rangle \leq \sup_{u\in\co F(x(t))} \inf_{v\in D(x(t))}\langle u,v \rangle \leq -\inf_{w\in D(x(t))} \kappa (w) \leq - \zeta.$$ 
Hence
$$    f(x(t_k)) = f(x(0)) - \kappa \int_0^{t_k}(f \circ x)'(t) dt \leq \ell -  \zeta t_k.$$ 
In \eqref{level_c}, we use the tracking bound \eqref{subeq:tracking} and rewrite the maximum into a minimum. Finally, \eqref{level_d} follows from the definition of $\eta$, namely, $\eta = \min\{\ell,\zeta T\}/(3L)\leq \ell/(2L)$. In particular, when $t_k \geq T/2$, \eqref{level_c} becomes
\begin{equation*} 
f(x_k) \leq L\eta + \ell - \min\{\ell/2,\zeta t_k\} \leq L\eta + \ell - \min\{\ell/2,\zeta T/2\} = \ell - \min\{\ell,\zeta T\}/6. \qedhere \end{equation*}
\end{proof}

We are now ready to prove our first main result. It forms the basis for ensuing results.

\begin{proof}[Proof of \cref{thm:stable_point}]
Without loss of generality, assume that $f(\overline x) =g(\overline x)=0$. Since $\overline{x}$ is a strict local minimum of $f+g$ and $f+g\leq 2\max\{f,g\}$, it is also a strict local minimum of $\widehat f=\max\{f,g\}$. Thus, for all sufficiently small $\epsilon>0$, we have $\widehat f(x)>0$ for all $x \in B_{2\epsilon}(\overline{x})\setminus \{\overline{x}\}$. Fix such an $\epsilon>0$ from now on and assume $f$ and $g$ are continuous on $B_{2\epsilon}(\overline x)$. Let $\ell=\min\{\widehat f(x)/2:~\epsilon/2\leq |x-\overline{x}|\leq \epsilon \}>0$ be such that $(0,\ell]$ is devoid of $D$-critical values of $f$, after possibly reducing $\epsilon$. Naturally, $[\widehat f\leq \ell]\cap {B}_{\epsilon}(\overline{x})\subseteq B_{\epsilon/2}(\overline{x})$ and $[\widehat f\leq \ell]\cap {B}_{\epsilon}(\overline{x})$ is bounded. After possibly reducing $\epsilon$ again, $g$ is d-Lyapunov on $B_\epsilon(\overline{x})$ with upper bound $\widehat \alpha>0$ on the step sizes. By continuity of $\widehat f$ on $B_\epsilon(\overline{x})$, there exists $\delta>0$ such that 
\begin{equation}
    \label{eq:inclusions}
      B_{\delta}(\overline{x})\subseteq [\widehat f\leq \ell]\cap {B}_{\epsilon}(\overline{x})\subseteq B_{\epsilon/2}(\overline{x}). 
\end{equation}
By applying \cref{lemma:descent} to $r=\epsilon/2$ and the bounded set $Y = B_\epsilon(\overline{x})$, there exist $T>0$ and $\overline{\alpha} \in (0,\min\{\widehat \alpha,T/2\}]$ such that, for all $\{\alpha_k\}_{k\in\N}\subseteq (0,\overline{\alpha}]$ and $\{x_k\}_{k\in\N}\subseteq \R\p n$ satisfying 
$$\forall k \in \N, ~~~ x_{k+1} \in x_k + \alpha_k F(x_k), ~~~ x_0 \in X = [\widehat f\leq \ell]\cap B_{\epsilon}(\overline{x}),$$
for all $k \in \N$, we have 
$$t_k \leq T~\Longrightarrow~ x_k \in B_{\epsilon/2}([\widehat f\leq \ell]\cap B_{\epsilon}(\overline{x})) \subseteq B_{\epsilon}(\overline{x}) ~~~ \land ~~ ~ T/2\leq t_k \leq T ~\Longrightarrow ~ f(x_k) \leq \ell.$$ 
This applies in particular when $x_0\in B_\delta(\overline{x})$ by \eqref{eq:inclusions}. Since $\alpha_k\leq \overline{\alpha}\leq\widehat\alpha$, if $t_k \leq T$, then $\ell\geq g(x_0)\geq\dots\geq g(x_k)$.
Let $K = \max \{ k \in \N: t_k \leq T\}$. If $K=\infty$, then the proof is done. Otherwise, since $\alpha_k \leq \overline{\alpha}\leq T/2$ for all $k\in\N$, one has $2 \leq K < \infty$ and $t_K \geq T/2$. Hence $x_K \in [f\leq \ell]\cap [g\leq\ell] \cap B_\epsilon(\overline{x}) = [\widehat f\leq \ell]\cap B_\epsilon(\overline{x})$. One now concludes by induction.
\end{proof}

\subsection{Stable sets}
    \label{subsec:Stable sets}


The d-Lyapunov requirement around a point in
\cref{thm:stable_point} can be too stringent in applications. When that happens, it is still possible to verify it in some regions around a set. This calls for extending the notion of d-stability from points to sets, mimicking the existing extension for stability \cite[Definition 8.1]{han2017attractors}.

\begin{definition}
\label{def:stable_set}
Given $F:\mathbb{R}^n \rightrightarrows \mathbb{R}^n$, a set $X \subseteq \mathbb{R}^n$ is d-stable if
$$ \forall \epsilon >0, ~ \exists \delta,\overline{\alpha}>0: ~ \forall \{\alpha_k\}_{k \in \N} \subseteq (0,\overline{\alpha}], ~~~ 
        x_0 \in B_{\delta}(X) ~~~ \Longrightarrow ~~~ 
        \{x_k\}_{k\in \N} \subseteq B_{\epsilon}(X) $$
for any $\{x_k\}_{k\in \N} \subseteq \mathbb{R}^n$ such that  $x_{k+1} \in x_k + \alpha_k F(x_k)$ for all $k\in\N$. 
\end{definition}

We also introduce the notion of asymptotic d-stability.

\begin{definition}
\label{def:asym_stable_set}
Given $F:\mathbb{R}^n \rightrightarrows \mathbb{R}^n$ and $p\in[1,\infty)$, a set $X \subseteq \mathbb{R}^n$ is asymptotically $p$-d-stable if it is d-stable and
$$ \exists \delta_0,\overline{\alpha}_0>0: ~ \forall \{\alpha_k\}_{k \in \N}\subseteq (0,\overline{\alpha}_0], ~~~ 
        x_0 \in B_{\delta_0}(X) ~~~ \Longrightarrow ~~~ d(x_k,X) \to 0 $$
for any $\{x_k\}_{k\in \N} \subseteq \mathbb{R}^n$ such that $x_{k+1} \in x_k + \alpha_k F(x_k)$ for all $k\in\N$ where $\sum_{k=0}\p \infty \alpha_k\p p=\infty$ and $\alpha_k\to 0$.
\end{definition}

\cref{thm:stable_point} admits the following extension.

\begin{theorem}
\label{thm:stable_set}
Under \cref{assume:Ff} with $\min f = 0$, suppose
\begin{enumerate}[label=\rm{(\rm{\roman*})}]
    \item $g:\R^n\to\eR$ is continuous near $\arg\min f$ with $\min g = 0$;
    \item $g$ is d-Lyapunov (resp. $p$-d-Lyapunov) near every point in $\arg\min f\cap [g>0]$ near $\arg\min f+g$;
    \item $\arg\min f+g$ is bounded; 
\end{enumerate}
Then $\arg\min f+g$ is d-stable (resp. asymptotically $p$-d-stable). 
\end{theorem}
\begin{proof}
Let $\widehat f=\max\{f,g\}$. Since $f$ and $g$ are continuous near the bounded set $\arg\min f+g$, $\widehat f$ is continuous near the bounded set $[\widehat f=0]$. For $\epsilon>0$ sufficiently small, let $$\ell=\min\{\widehat f(x)/2:~\epsilon/2\leq d(x,[\widehat f=0])\leq \epsilon \}\in (0,\overline{\ell}).$$ Naturally, $[\widehat f\leq \ell]\cap B_{\epsilon}([\widehat f=0])\subseteq B_{\epsilon/2}([\widehat f=0])$. After possibly reducing $\epsilon$, $g$ is d-Lyapunov (resp. $p$-d-Lyapunov) near every point in $[f=0]\cap [0<g\leq \ell]\cap B_{\epsilon}([\widehat f=0])$. Since $[f = 0]\cap [\ell/2\leq g\leq \ell]\cap B_{\epsilon}([\widehat f=0])$ is bounded, there exists $\rho>0$ such that $g$ is d-Lyapunov (resp. $p$-d-Lyapunov) on $B_\rho( [f=0]\cap [\ell/2\leq g\leq \ell]\cap B_{\epsilon}([\widehat f=0]) )$ with upper bound $\widehat \alpha>0$ on the step sizes. Again by boundedness, there exists $\ell'\in(0,\ell)$ such that 
$$   [f\leq 3\ell'/2]  \cap [\ell/2\leq g\leq \ell]\cap B_\epsilon([\widehat f=0])\subseteq      B_\rho ([f=0]\cap [\ell/2\leq g\leq \ell]\cap B_{\epsilon}([\widehat f=0]) ).$$
If not, one readily obtains a contradiction using a sequence $x_k\in [f\leq \ell/(k+1)]  \cap [\ell/2\leq g\leq \ell]\cap B_\epsilon([\widehat f=0]) \subseteq B_{\epsilon/2}([\widehat f=0])$ which must have a limit point $\overline{x}\in [f=0]  \cap [\ell/2\leq g\leq \ell]\cap B_\epsilon([\widehat f=0])$. 

By continuity of $\widehat f$, there exists $\delta>0$ such that 
\begin{equation}
    \label{eq:inclusions_set}
      B_{\delta}([\widehat f=0])\subseteq [f\leq \ell']  \cap [ g\leq \ell]\cap B_\epsilon([\widehat f=0])\subseteq B_{\epsilon/2}([\widehat f=0]).
\end{equation}
By applying \cref{lemma:descent} to $r=\epsilon
/2$ and the bounded set $Y = B_\epsilon([\widehat f=0])$, there exist $T>0$, $\kappa\in \cP$, and $\overline{\alpha} \in (0,\min\{\widehat \alpha,T/2\}]$ such that, for all $\{\alpha_k\}_{k\in\N}\subseteq (0,\overline{\alpha}]$ and $\{x_k\}_{k\in\N}\subseteq \R\p n$ satisfying 
$$\forall k \in \N, ~~~ x_{k+1} \in x_k + \alpha_k F(x_k), ~~~ x_0 \in X = [f\leq \ell']  \cap [ g\leq \ell]\cap B_\epsilon([\widehat f=0]),$$
for all $k\in\N$, we have 
$$
\begin{array}{ccc}
    t_k \leq T & \Longrightarrow & x_k \in [f\leq 3\ell'/2]\cap B_\epsilon([\widehat f=0]), \\[2mm] 
    T/2\leq t_k \leq T & \Longrightarrow & f(x_k) \leq \ell'- \min\{\ell',\zeta T\}/6,
\end{array}
$$
where $\zeta = \inf \{ \kappa(w) : w \in D(B_r(X) \cap [\ell'/2 \leq f \leq \ell'])\}$. This applies in particular when $x_0\in B_{\delta}([\widehat f=0])$ by \eqref{eq:inclusions_set}. 

Suppose $t_k \leq T$. If $\ell/2\leq g(x_k)\leq \ell$, then $g(x_{k+1})\leq g(x_k)$ since $g$ is d-Lyapunov on $[f\leq 3\ell'/2]  \cap [\ell/2\leq g\leq \ell]\cap B_\epsilon([\widehat f=0])$ and $\alpha_k\leq \overline{\alpha}\leq\widehat\alpha$ (resp. $g(x_{k+1})\leq g(x_k) - \omega\alpha_k\p p$ if $g$ is $p$-d-Lyapunov). If $g(x_k)<\ell/2$, then $g(x_{k+1})=g(x_k)+g(x_{k+1})-g(x_k)\leq \ell/2+\ell/6=2\ell/3 \leq \ell$. The last inequality holds after possibly reducing the upper bound $\overline{\alpha}$ on the step sizes, since $g$ is uniformly continuous on $B_{\epsilon}([\widehat f=0])$. Namely, there exists $\widetilde\alpha>0$ such that 
$$
    \forall \alpha\in (0,\widetilde\alpha],~\forall x\in B_{\epsilon}([\widehat f=0]),~\forall v\in F(x), \quad |g(x+\alpha v)-g(x)|\leq \ell/6.
$$

Let $K = \max \{ k \in \N: t_k \leq T\}$. If $K=\infty$, then the proof of stability is done. Otherwise, since $\alpha_k \leq \overline{\alpha}\leq T/2$ for all $k\in\N$, one has $2 \leq K < \infty$ and $t_K \geq T/2$. Hence $f(x_K) \leq \ell'-\min\{\ell',\zeta T\}/6 \leq \ell'$ and $x_K \in [f\leq \ell']  \cap [ g\leq \ell]\cap B_\epsilon([\widehat f=0])$. Stability now follows by induction. Thus
$$\forall \{\alpha_k\}_{k\in\N} \subseteq (0,\overline{\alpha}], ~~~ x_0 \in B_{\delta}([\widehat f=0]) ~~~ \Longrightarrow ~~~ \{x_k\}_{k\in\N} \subseteq B_{\epsilon}([\widehat f=0])$$
for any $\{x_k\}_{k\in \N} \subseteq \mathbb{R}^n$ such that $x_{k+1} \in x_k + \alpha_k F(x_k)$ for all $k\in \N$.

As for asymptotic stability, observe that $\sum_{k=0}\p\infty\alpha_k\p p = \infty$ with $p\in[1,\infty)$ implies $\sum_{k=0}\p\infty\alpha_k = \infty$. Thus $t_k\to \infty$, $K < \infty$, and $f(x_K) \leq \ell'-\zeta T/6$. Since $\alpha_k\to 0$, one may eventually apply \cref{lemma:descent} to $\ell'-\min\{\ell',\zeta T\}/6$ and $[f\leq \ell'-\min\{\ell',\zeta T\}/6]  \cap [ g\leq \ell]\cap B_\epsilon([\widehat f=0])$ with the same $r$, $Y$, $T$, and $\kappa$, implying that $f(x_k)\leq 3(\ell'-\min\{\ell',\zeta T\}/6)/2$ for all sufficiently large $k\in \N$. Let $\{\zeta_i\}_{i\in\N}$ and $\{\ell_i'\}_{i\in\N}$ be defined by $\zeta_0=\zeta$, $\ell'_0 = \ell'$, 
$$\zeta_i = \inf \{\kappa(v) : v \in D(B_r(X) \cap [\ell_i'/2 \leq f \leq \ell_i'])\},$$ 
and $\ell'_{i+1} = \ell_i'-\min\{\ell_i',\zeta_i T\}/6$. By induction, $f(x_k) \leq 3\ell'_i/2$ eventually holds for all $i\in\N$. 

Since $\ell_i'$ is decreasing and positive, it has a limit $\overline \ell'\in [0,\ell'] \subseteq [0,\ell) \subseteq (0,\overline{\ell})$. As $F$ is locally bounded by \cref{fact:usc_bounded}, and $\kappa\in \cP$, $\zeta_i$ is bounded and admits a limit point $\overline\zeta$. Passing to the limit yields $\overline\ell' = \overline\ell' - \min\{\overline\ell',\overline\zeta T\}/6$, i.e., $\min\{\overline\ell',\overline\zeta T\}=0$. 
As $(0,\overline{\ell})$ is devoid of $D$-critical values of $f$ by \cref{assume:Ff} \ref{item:f}, $\overline\zeta>0$ if $\overline{\ell}'>0$, and so $\overline\ell'=0$. Since $\ell_i'\to 0$, we may improve our earlier statement: $f(x_k) \leq \ell'_i$ eventually holds for all $i\in\N$.

On the other hand, $g(x_k) \leq 2\ell/3$ eventually holds since $\alpha_0\p p + \alpha_1\p p + \cdots = \infty$. After possibly reducing $\rho$ and $\widehat \alpha$, $g$ is $p$-d-Lyapunov on $B_\rho( [f=0]\cap [\ell/3\leq g\leq 2\ell/3]\cap B_{\epsilon}([\widehat f=0]))$. As $\ell_i'\to 0$, eventually
$$[f\leq 3\ell_i'/2]  \cap [\ell/3\leq g\leq 2\ell/3]\cap B_\epsilon([\widehat f=0])\subseteq      B_\rho ([f=0]\cap [\ell/3\leq g\leq 2\ell/3]\cap B_{\epsilon}([\widehat f=0]) ).$$
Applying \cref{lemma:descent} to such $\ell'_i$ and $X_i = [f\leq \ell_i']  \cap [\ell/3\leq g\leq 2\ell/3]\cap B_\epsilon([\widehat f=0])$ with the same $r$, $Y$, $T$, and $\kappa$, we find that $g(x_k) \leq 4\ell/9$ eventually holds. By induction, $g(x_k) \leq (2/3)\p j\ell$ eventually holds for all $j\in\N$. Since $f(x_k)\to 0$, $g(x_k)\to 0$, and $x_k\in B_{\epsilon}([\widehat f=0])$ for all $k\in \N$, we conclude that $d(x_k,[\widehat f=0])\to 0$.
\end{proof}

\subsection{Attractors}
\label{subsec:attractors}
The notion of d-stability of a set can be strengthened as follows.

\begin{definition}
\label{def:attractor}
    Given $F:\mathbb{R}^n \rightrightarrows \mathbb{R}^n$, $f:\mathbb{R}^n \rightarrow \R_+$, and $p\in [1,\infty)$, a bounded set $A \subseteq [f=0]$ is a $p$-attractor if for any $\epsilon,\ell>0$ and any bounded set $X \subseteq [f=0]$, 
$$
    \exists \delta,\overline{\alpha}>0:~ \forall \{\alpha_k\}_{k\in\N} \subseteq (0,\overline{\alpha}],~ \forall x_0 \in B_{\delta}(X), ~~~ f(x_k)\leq \ell ~~\land~~ \exists k_0 \in \mathbb{N}: \{x_k\}_{k\geq k_0} \subseteq B_{\epsilon}(A)
$$
for any $\{x_k\}_{k\in \mathbb{N}}\subseteq \R\p n$ such that $x_{k+1} \in x_k + \alpha_k F(x_k)$ for all $k \in \mathbb{N}$ where $\sum_{k=0}\p \infty \alpha_k\p p=\infty$. 
It is an asymptotic $p$-attractor if in addition it is asymptotically $p$-d-stable. 
\end{definition}

The above terminology is an allusion to the notion of attractor in continuous-time dynamics \cite[Definition 1.8]{han2017attractors}, although it differs in several regards: $A$ need not be closed nor invariant, $X$ is not some fixed neighborhood of $A$, nor can it be taken to be any bounded subset of $\R\p n$ (as in globally uniformly attracting sets \cite[Definition 8.2]{han2017attractors}), and $k_0$ can depend on $x_0$. Besides, the requirement on the minimal rate of decrease of the step sizes, absent from previous works, is crucial in our analysis. \cref{thm:stable_set} admits the following consequence. 

\begin{theorem}
\label{thm:attractor}
Under \cref{assume:Ff} with $\min f = 0$, suppose
\begin{enumerate}[label=\rm{(\rm{\roman*})}]
    \item $g:\R^n\to\R$ is continuous with $\min g = 0$;
    \item $g$ is $p$-d-Lyapunov near every point in $\arg\min f\cap [g>0]$;
    \item $f+g$ is coercive.
\end{enumerate}
Then $\arg\min f+g$ is an asymptotic $p$-attractor.
\end{theorem}
\begin{proof}
Let $\epsilon,\ell>0$, $X \subseteq [f=0]$ be bounded, and $\widehat{g} = \sup_X g$. Without loss of generality, assume $\widehat g > 0$. Since $f+g$ is coercive and $f,g$ are continuous, after possibly reducing $\ell$, we have $\ell\in (0,\widehat g)$ and
$[f\leq \ell]\cap[g\leq \ell] \subseteq B_\epsilon([f=0]\cap[g=0])$.
Also, $[f=0]\cap[\ell/2\leq g\leq \widehat g +\ell]$ is bounded. Thus there exists $\rho>0$ such that $g$ is $p$-d-Lyapunov on $B_\rho([f=0]\cap[\ell/2\leq g\leq \widehat g+\ell])$. Again by continuity and coercivity, there exists $\ell'\in (0,\ell)$ such that
$$[f \leq\ell']\cap[\ell/2\leq g\leq \widehat g+\ell] \subseteq B_\rho([f=0]\cap[\ell/2\leq g\leq \widehat g+\ell]).$$
By continuity, there exists $\epsilon'>0$ such
$$B_{\epsilon'}([f=0]\cap[g\leq \widehat{g}]) \subseteq [f\leq \ell']\cap[g\leq \widehat{g}+\ell].$$
By \cref{thm:stable_set}, $[f=0]\cap[g\leq \widehat{g}]$ is d-stable. Recall that $X\subseteq[f=0]\cap[g\leq \widehat{g}]$. Thus there exists $\delta>0$ such that, after possibly reducing $\overline{\alpha}$, $$\forall \{\alpha_k\}_{k\in\N} \subseteq (0,\overline{\alpha}],~~~ x_0 \in B_\delta(X) ~~~ \Longrightarrow ~~~ x_k \in B_{\epsilon'}([f=0]\cap[g\leq \widehat{g}])$$
for any $\{x_k\}_{k\in \N}\subseteq \R\p n$ such that $x_{k+1}\in x_k +\alpha_k F(x_k)$ for all $k\in \N$. Hence $f(x_k) \leq \ell' \leq \ell$ and $g(x_k) \leq \widehat g + \ell$. If $g(x_k) \geq \ell/2$, then $g(x_{k+1}) \leq g(x_k) - \omega \alpha_k\p p$. Otherwise, by uniform continuity (as in the proof of \cref{thm:stable_set}), we have $g(x_{k+1}) = g(x_{k+1}) - g(x_k) + g(x_k) \leq \ell/2 + \ell/2 \leq \ell$. Since $\sum_{k=0}\p \infty\alpha_k\p p = \infty$, $g(x_k) \leq \ell$ eventually holds, at which point $x_k \in [f\leq \ell]\cap[g\leq\ell] \subseteq B_\epsilon([f=0]\cap[g=0])$. Thus $[f=0]\cap[g=0]=\arg\min f+g$ is an attractor. If $\alpha_k \geq \underline{\alpha} >0$ for some $\underline{\alpha}>0$ for all $k\in \N$, then $g(x_k) \leq \ell$ eventually holds after at most $k_0 = \lceil \widehat g/(\omega \underline{\alpha}\p p)\rceil$ iterations. If $\alpha_k\to 0$, then $\arg\min f+g$ is asymptotically $p$-d-stable by \cref{thm:stable_set}, and thus an asymptotic $p$-attractor.
\end{proof}

The results in \cref{sec:Theory} admit several extensions.


\begin{remark}[Continuity and coercivity]
\label{rem:continuity_coercivity}
    In \cref{thm:stable_set} (resp. \cref{thm:attractor}), it suffices if $g:\R^n\to\eR$ is continuous on $U\setminus \arg\min f+g$ where $U$ is a neighborhood of $\arg\min f+g$ (resp. if $g:\R^n\to\R_+$ is continuous on $U\setminus \arg\min f+g$ where $U$ is a neighborhood of $[f=0]\cap [g\leq \widetilde{g}]$ for any $\widetilde g \in (0,\sup_{[f=0]} g)$).
    In \cref{thm:attractor}, instead of coercivity of $f+g$, it suffices if $g([f=0]) \subseteq (-\infty,\sup_{[f=0]} g)$ and for all $\widetilde g \in (0,\sup_{[f=0]} g)$, there exists $\ell>0$ such that $[f\leq \ell]\cap[g\leq \widetilde{g}]$ is bounded.
    \end{remark}
\begin{remark}[Almost sure stability]
    \label{rem:almost_sure}
    \cref{def:stable_set,def:asym_stable_set,def:attractor} (resp. \cref{def:pdL}) can be loosened to hold for almost every initial point $x_0$ (resp. $x\in U$). \cref{thm:stable_point,thm:stable_set,thm:attractor} generalize in the obvious way.
    In \cref{thm:attractor}, if $g:\R\p n \to \eR$ is continuous on a full measure domain, and every point in $\arg\min f\cap [0<g<\infty]$ admits a neighborhood where $g$ is $p$-d-Lyapunov, then one obtains an almost sure asymptotic $p$-attractor.
\end{remark}
\begin{remark}[Multistep d-Lyapunov functions]
    \label{rem:multi}
    Suppose one requires $\sum_{k=0}\p \infty \alpha_{kq+i}\p  p=\infty$ for all $i\in\llbracket 0,q-1\rrbracket$ for some $q\in \N\p *$ in \cref{def:asym_stable_set,def:attractor} instead of $\sum_{k=0}\p \infty \alpha_k\p  p=\infty$. This requirement still fulfilled by $\alpha_k = \beta/(k+1)\p{1/\gamma}$ with $\beta>0$ and $\gamma\geq p$.
    Then one may relax the $p$-d-Lyapunov assumption in \cref{thm:stable_set,thm:attractor} to $(p,q)$-d-Lyapunov. The proof of the former can be applied to a single subsequence $\{x_{kq+i}\}_{k\in \N}$, yielding $\{x_{kq+i}\}_{k\in \N}\subseteq B_{\epsilon}([\widehat f = 0])$ or $d(x_{kq+i},[\widehat f = 0])\to 0$. Since $\alpha_k\to 0$ and $F$ is locally bounded, boundedness and convergence extend to the whole sequence. Or one may treat the $q$ subsequences simultaneously.
\end{remark}

\section{Application}
\label{sec:Application}

This section applies the theory of d-stability to the Bouligand subgradient method. We first recall some basic notions from variational analysis. Given $f:\mathbb{R}^n\rightarrow \overline{\mathbb{R}}$ and a point $\overline{x}\in\mathbb{R}^n$ where $f(\overline{x})$ is finite, the regular subdifferential, subdifferential, horizon subdifferential \cite[Definition 8.3]{rockafellar2009variational}, and Clarke subdifferential of $f$ at $\overline{x}$ \cite[Definition 4.1]{drusvyatskiy2015curves} are respectively given by
\begin{gather*}
    \widehat{\partial} f (\overline{x}) = \{ v \in \mathbb{R}^n : f(x) \geq   f(\overline{x}) + \langle v , x - \overline{x} \rangle + o(|x-\overline{x}|) ~\text{near}~ \overline{x} \}, \\
    \partial f(\overline{x}) = \{ v \in \mathbb{R}^n : \exists (x_k,v_k)\in \gph\hspace*{.5mm}\widehat{\partial} f: (x_k, f(x_k), v_k)\rightarrow(\overline{x}, f(\overline{x}), v) \}, \\[1mm]
    \partial^\infty f(\overline{x}) = \{ v \in \mathbb{R}^n : \exists (x_k,v_k)\in \gph\hspace*{.5mm}\widehat{\partial} f: \exists \tau_k \searrow 0: (x_k, f(x_k), \tau_kv_k)\rightarrow(\overline{x}, f(\overline{x}), v) \}, \\[2mm]
    \overline{\partial} f(\overline{x}) = \cco [\partial f(\overline{x}) + \partial^\infty f(\overline{x})].
\end{gather*}
Given $f:\R\p n\to \R$, the Bouligand subdifferential at $\overline{x}\in \R\p n$ is defined by
$$\overline{\nabla} f(\overline{x}) = \{ v \in \mathbb{R}^n : \exists x_k \xrightarrow[\Omega]{} \overline{x} ~\text{with}~ \nabla f(x_k) \rightarrow v\}$$
where $\Omega$ are the differentiable points of $f$. The letter $\Omega$ under the arrow means $x_k \in \Omega$ and $x_k \rightarrow \overline{x}$. 

\begin{definition}
    \label{def:BS}
    Given $f:\R\p n \to \R$, the Bouligand subgradient method is the Euler discretization of $-\overline\nabla f$, whose trajectories $\{x_k\}_{k\in \N}\subseteq \R\p n$ obey
$$\forall k \in \N, ~~~ x_{k+1} \in x_k - \alpha_k \overline{\nabla} f(x_k),$$
for some sequence $\{\alpha_k\}_{k\in \N} \subseteq (0,\infty)$.
\end{definition}

Contrary to the subdifferential $\partial f$ which satisfies Fermat's rule \cite[Theorem 10.1]{rockafellar2009variational}, if $\overline{x}$ is a local minimum of $f$, then we may have $0\notin \overline\nabla f(\overline{x})$. This happens if the norm of nearby gradients is bounded away from zero. The Bouligand subdifferential is thus the right tool for analyzing d-stability of local minima with this property, slightly stronger than sharpness. It prevents the dynamics from getting stuck at a local minimum.
Of course, one could pander to avoidance results instead \cite[Theorem 7, Claim 3]{bolte2020mathematical}, but this seems like a needless detour.

The Bouligand subdifferential satisfies the key assumption in the theory of d-stability.

\begin{proposition}
\label{prop:Bouligand}
    If $f:\R\p n\to \R$ is locally Lipschitz semi-algebraic, then
    \cref{assume:Ff} thus holds with $F = -\overline{\nabla} f$ and $D = \overline{\partial} f$.
\end{proposition}
\begin{proof}
    Since $f$ is locally Lipschitz, $\overline{\nabla} f$ is locally bounded with nonempty values and has a closed graph by \cite[Theorem 9.61]{rockafellar2009variational}. By \cref{fact:usc_bounded}, it is upper semicontinuous with nonempty compact values. Since $f$ is locally Lipschitz semi-algebraic, it has finitely many $D$-critical values by the semi-algebraic Morse-Sard theorem \cite[Corollary 9]{bolte2007clarke} and $D$ is a conservative field for $f$ by \cite[Corollary 2, Proposition 2]{bolte2020conservative} (see also \cite[Corollary 5.4]{drusvyatskiy2015curves}). Again due to \cite[Theorem 9.61]{rockafellar2009variational}, $\co F(x) = -\co \overline{\nabla} f(x) = -\overline{\partial} f(x) = D(x)$ for all $x\in\R\p n$. Thus 
    $$\max_{u\in\co F(x)}\min_{v\in D(x)}\langle u , v \rangle = \max_{u\in D(x)}\min_{v\in D(x)}\langle -u , v \rangle = -\min_{u\in D(x)}\max_{v\in D(x)}\langle u , v \rangle \leq -\min_{u\in D(x)} |u|\p 2$$
    and one may choose $\kappa(x)=|x|\p 2$.
\end{proof}

Complementing \cref{eg:d-stable}, we make the connection between stability and d-stability for the Bouligand subdifferential.

\begin{proposition}
    Let $f:\R\p n \to \R$ be locally Lipschitz semi-algebraic and $\overline{x}\in \R\p n$. The following hold:
    \begin{enumerate}[label=\rm{(\roman{*})}]
        \item $\overline{x}$ is stable for $-\overline\partial f$ iff $\overline{x}$ is asymptotically stable for $-\overline\partial f$ iff $\overline{x}$ is a local minimum of $f$. \label{item:stable}
        \item If $\overline{x}$ is d-stable for $-\overline\nabla f$, then $\overline{x}$ is stable for $-\overline\partial f$. \label{item:d-stable}
    \end{enumerate}
\end{proposition}
\begin{proof}
    \ref{item:stable} This can be proved in a similar flavor to \cite[Theorem 3]{absil2006stable} (see also \cite[Lemma 1]{josz2025reachability}). As for \ref{item:d-stable}, if $\overline{x}$ is d-stable for $-\overline \nabla f$, then it is a local minimum of $f$. This can be shown using the same arguments as in \cite[Theorem 1]{josz2023lyapunov} modulo \cref{lemma:descent}. One then concludes using \ref{item:stable}.
\end{proof}

We conclude with two examples. The Bouligand sign of a scalar $x$ is defined by
\begin{equation*}
    \sgn(x)=\left\{\begin{array}{cl}
        x/|x| & \text{if}~ x\neq 0, \\
        \{-1,1\} & \text{if}~ x=0.
    \end{array}\right.
\end{equation*}
\begin{example}
    \label{eg:ellipse}
    The d-stable global minima of $f(x,y) = |ax^2+by^2-1|$  are $\pm (0,1/\sqrt{b})$ if $a>b>0$ or $a<0<b$. They are almost sure asymptotic 2-attractors.
\end{example}
\begin{proof}
Compute
$$\overline{\nabla} f(x,y) = 2\sgn(ax\p 2 + by\p 2-1) \begin{pmatrix}
    ax \\ by
\end{pmatrix}.$$
We omit the trivial case where $ab=0$. 
Let $g:\R\p 2\to \eR$ be defined by
$$g(x,y) = \left\{
\begin{array}{cc}
    |x|\p b/|y|\p a & \text{if}~y\neq 0, \\
    \infty & \text{if}~y=0.
\end{array}
\right.
$$ 
Let $(\widetilde{x},\widetilde{y})\in[0<g<\infty]$.
For all $(x,y)$ near $(\widetilde{x},\widetilde{y})$, for all $s\in \sgn(ax^2+by^2-1)$, and for all sufficiently small $\alpha>0$, we have
\begin{align*}
    g(x\p+,y\p+) &= \frac{|x\p +|\p b}{|y\p +| \p a} = \frac{\left|x-\alpha sax\right|^b}{\left|y-\alpha sby\right|^a}=\frac{|x|^b}{|y|^a}\frac{\left(1-\alpha sa\right)^b}{\left(1-\alpha sb\right)^a} \\
    &=g(x,y)\exp\left[b\ln\left(1-\alpha sa\right)-a\ln\left(1-\alpha sb\right)\right]\\
    &=g(x,y)\exp\left(-s\p 2(ba^2-ab^2)\alpha^2/2 +O(\alpha^3)\right) \\
    &\leq g(x,y)\exp\left( -ab(a-b)\alpha\p 2/4\right) \\
    &\leq g(x,y)-g(x,y)ab(a-b)\alpha\p 2/4 \\
    &\leq g(x,y)-g(\widetilde{x},\widetilde{y})ab(a-b)\alpha\p 2/8
\end{align*}
since $\ln(1+t) = t-t\p 2/2+O(t\p 3)$.
Thus $g$ is $2$-d-Lyapunov near $(\widehat{x},\widehat{y})$. 

For pedagogical purposes, we successively apply the three main results in this paper, which yield increasingly strong conclusions. If $f(x,y)+g(x,y)=0$, then $f(x,y)=g(x,y)=0$ and $(x,y)=\pm(0,1/\sqrt{b})$. These are hence strict local minima of $f+g$. By \cref{thm:stable_point}, $\pm(0,1/\sqrt{b})$ are d-stable. By \cref{thm:stable_set}, they are asymptotically 2-stable.
If $a>b>0$, then $f$ is coercive, and $g$ is nonnegative, so $f+g$ is coercive. If $a<0<b$, then $f(x,y)+g(x,y)\geq g(x,y)=|x|^b|y|^{-a}\to\infty$ as $|x|,|y|\to\infty$, and $f(x,y)+g(x,y)\geq f(x,y)\to\infty$ as $|x|\to \infty$ and $y$ is bounded or vice-versa. Thus $f+g$ is coercive. By \cref{thm:attractor} and \cref{rem:almost_sure}, $\arg\min f+g = \{\pm(0,1/\sqrt{b})\}$ is an almost sure asymptotic 2-attractor.
\end{proof}

\begin{figure}[ht]
\centering
\begin{subfigure}{.49\textwidth}
  \centering
  \includegraphics[width=1\textwidth]{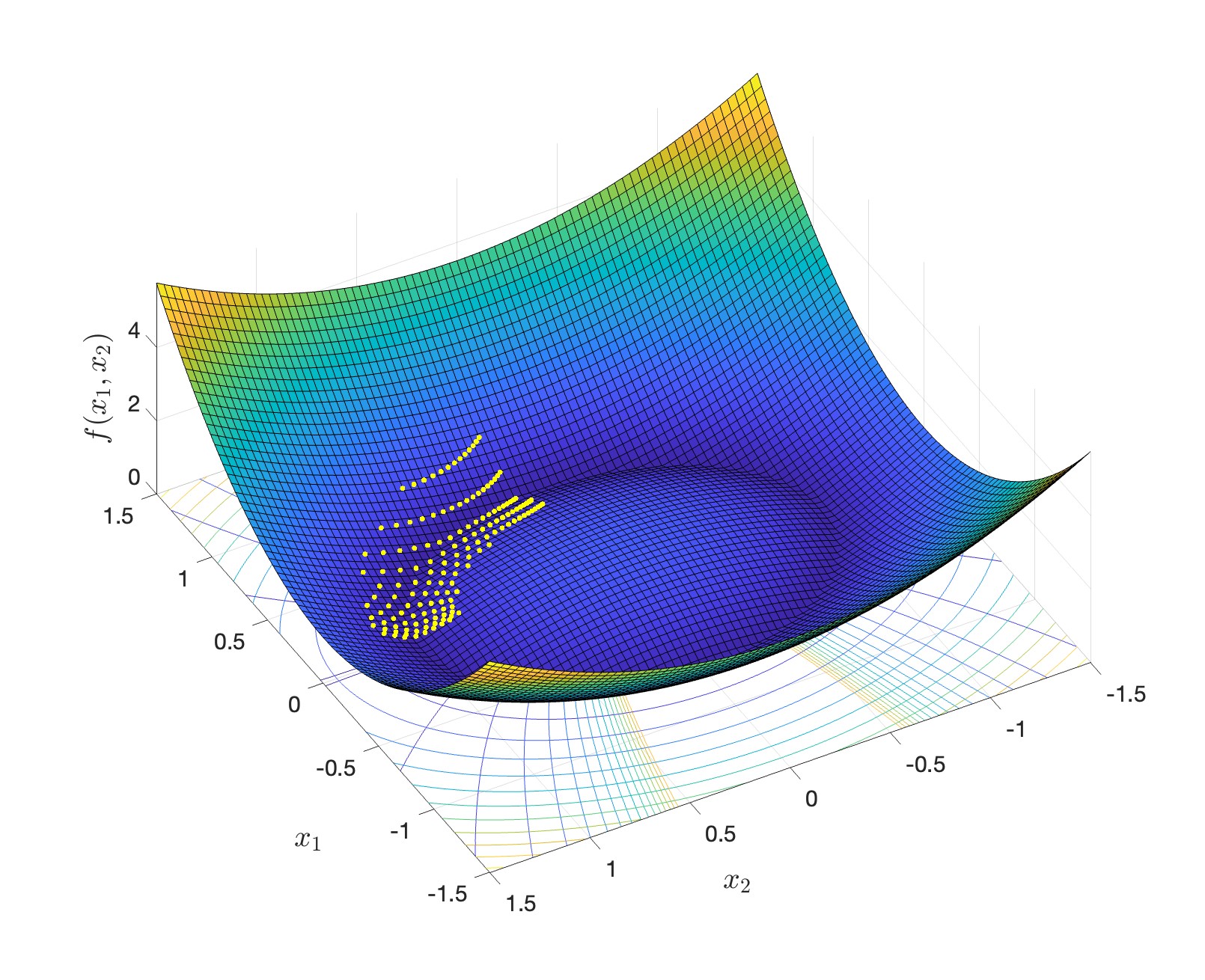}
  \caption{$f(x,y)=|2x^2+y^2-1|$}
  \label{fig:ellipse}
\end{subfigure}
\begin{subfigure}{.49\textwidth}
\centering
  \includegraphics[width=1\textwidth]{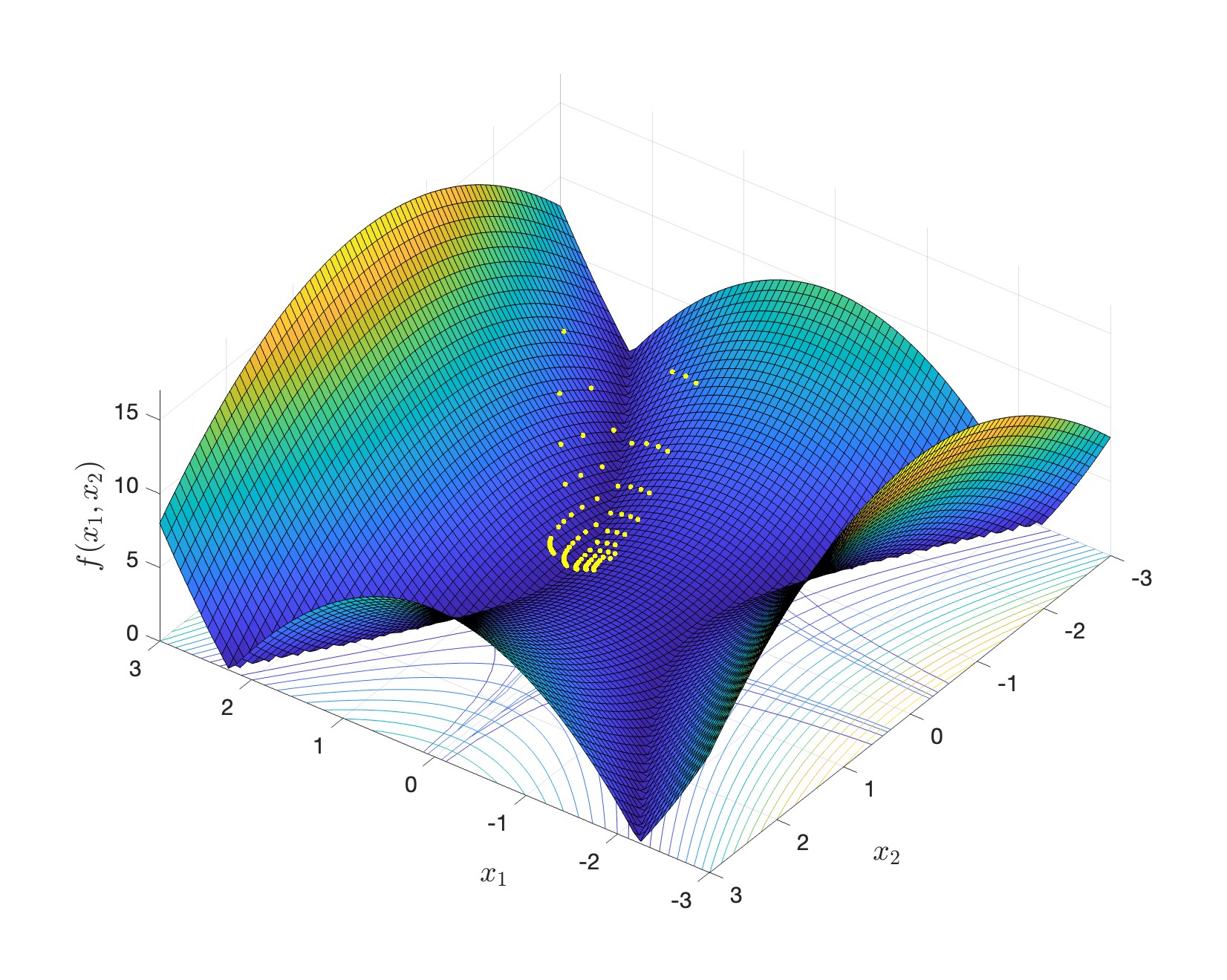}
  \caption{$f(x,y)=|2x^2-y^2-1|$}
  \label{fig:hyperbola}
\end{subfigure}
\caption{Bouligand subgradient method with constant step size.}
\end{figure}

\begin{example}
    \label{eg:mf1_ab_stable}
    The subset of global minima of $f(x) = |x_1x_3-a|+|x_2x_3-b|$ where $(a,b) \neq (0,0)$, $a\neq b$, given by
    $$
        X = \left\{ (at,bt,1/t) : \sqrt{ \frac{\sqrt{2}}{|a|+|b|} } \leqslant |t| \leqslant \sqrt{ \frac{\sqrt{2}}{||a|-|b||}} \right\}
    $$
    is an asymptotic 2-attractor.
\end{example}
    \begin{proof} 
    By \cite[Theorems 8.9 and 10.6]{rockafellar2009variational}, 
    \begin{equation*}
       \partial f(x) = \left\{ \begin{pmatrix}
            \lambda_1 x_3 \\
            \lambda_2 x_3 \\
            \lambda_1 x_1 + \lambda_2 x_2
        \end{pmatrix}
        ~,\hspace{3mm} \lambda_1\in\sign(x_1x_3-a),~ \lambda_2 \in \sign(x_2x_3-b)
        \right\}.
    \end{equation*}
    Let $x_t=(a,bt,1/t)$ with $t\neq 0$ be an arbitrary global minimum of $f$ and compute
    $$
        \overline{\nabla} f(x_t) = \left\{ 
        \pm \begin{pmatrix}
            1/t \\
            1/t \\
            at + bt
        \end{pmatrix}
        ~,~
        \pm \begin{pmatrix}
            \hphantom{-}1/t \\
            -1/t \\
            at - bt
        \end{pmatrix}
        \right\}.
    $$
Consider the function $g:\R\p 3\to\R$ defined by
    $$g(x) = \max\left\{0,\left[x_1^2+x_2^2 - x_3^2-(c+d)/2\right]\p2-(c-d)\p 2/4\right\}$$
    where 
    $$c=\sqrt{2}\frac{a\p 2+b\p 2}{|a|+|b|}-\frac{|a|+|b|}{\sqrt{2}} ~~~\text{and}~~~ d=\sqrt{2}\frac{a\p 2+b\p 2}{|a|-|b|}-\frac{|a|-|b|}{\sqrt{2}}.$$
    One can verify that $\arg\min f+g = X$.
    Since
    $$\forall v \in \overline{\nabla} f(x_t),~~~ \langle \nabla\p 2 g(x_t)v,v\rangle  = [(a\p 2+b\p 2)t^2 - 1/t^2 -(c+d)/2][2/t^2 - (a \pm b)\p 2t^2] < 0$$
    whenever $g(x_t)>0$, $g$ is 2-d-Lyapunov near such $x_t$ by \cref{prop:sufficient_dL2}. As $f+g$ is coercive, $X$ is an asymptotic 2-attractor by \cref{thm:attractor}. In particular, if $(a,b)=(0,1)$, then $X=\{\pm (0,2^{1/4},2^{-1/4})\}$. 
\end{proof}

A host of other examples will be given in \cite{josz2026implicit} with normalized gradient descent.

\section{Appendix}

\begin{proof}[Proof of \cref{thm:lyapunov}]
    Without loss of generality, $f(\overline{x})=0$.
    Let $\epsilon>0$ be sufficiently small so that $f(x)>0$ for all $x \in B_{3\epsilon}(\overline{x})$. Since $f$ is continuous, $\ell = \min \{f(x):\epsilon\leq |x-\overline{x}|\leq 2\epsilon\}>0$ and there exists $\delta>0$ such that $B_\delta(\overline{x})\subseteq [f\leq \ell/2]\cap B_{2\epsilon}(\overline{x}) \subseteq B_{\epsilon}(\overline{x})$. Suppose $x:[0,T)\rightarrow \mathbb{R}^n$ is an absolutely continuous function with $T\in(0,\infty]$ such that 
    $$\forae t \in (0,T), ~~~ x'(t) \in F(x(t)), ~~~ x(0) \in B_\delta(\overline{x}).$$ 
    Since $f$ is a potential for $D$,
    $$\forae t \in (0,T), ~~~ (f\circ x)'(t) = \min_{v \in D(x(t))} \langle x'(t) , v \rangle \leq \sup_{u\in F(x(t))} \min_{v \in D(x(t))} \langle u , v \rangle \leq 0.$$ 
    As $f$ is locally Lipschitz and $x(\cdot)$ is absolutely continuous, $f\circ x$ is absolutely continuous. Thus $$\forall t\in (0,T),~~~ f(x(t))-f(x(0)) = \int_0\p t (f\circ x)'(s)ds \leq 0,$$ and so $f(x(t)) \leq f(x(0))\leq \ell/2$. Recall that $|x(0)-\overline{x}|\leq \delta$. If there exists $t\in (0,T)$ such that $|x(t)-\overline{x}|>\epsilon$, then by the intermediate value theorem there exists $t_0\in(0,t]$ such that $\epsilon \leq |x(t_0)-\overline{x}|\leq 2\epsilon$, in which case $f(x(t_0))\geq \ell>\ell/2$, a contradiction.

    Assume \ref{item:negative_definite} holds and that $x(\cdot)$ is a maximal solution. Since it is bounded, it must be globally defined. We have
    $$\forae t \in (0,\infty), ~~~ (f\circ x)'(t) \leq \sup_{u\in F(x(t))} \min_{v \in D(x(t))} \langle u , v \rangle \leq -\inf_{w\in D(x(t))} \kappa(w).$$
    After possibly reducing $\epsilon>0$, $[0<f\leq \ell]$ is devoid of $D$-critical values.  
    Since $f \circ x$ is decreasing and nonnegative, it has a limit $f\p *$. Suppose $f\p *>0$ and let 
    $$\zeta = \inf\{ \kappa(w) : w \in D(x),~ |x-\overline{x}|\leq \epsilon,~ f\p * \leq f(x) \leq \ell \}.$$
    If $\zeta = 0$, then there exists $(x_k,w_k)\in \gph D$ such that $\kappa(w_k) \to 0$, $|x_k-\overline{x}|\leq \epsilon$, and $f\p * \leq f(x_k) \leq \ell$. Since $D$ is locally bounded, $(x_k,w_k)$ is bounded and thus has a limit point $(\widehat x , \widehat w)$. As $f$ is continuous, $D$ has a closed graph, and $\kappa\in\cP$, passing to limit yields $\kappa(\widehat w)=0$, $|\widehat x-\overline{x}|\leq \epsilon$, $0<f\p * \leq f(\widehat x) \leq \ell$, and $(\widehat x , \widehat w) \in \gph D$. This yields the contradiction $\widehat w = 0 \in D(\widehat x)$. Thus $\zeta >0$ and 
    $$\forall t\in (0,\infty),~~~ f(x(t))-f(x(0)) \leq - \int_0\p t \inf_{w\in D(x(s))} \kappa(w) ds \leq -\zeta t \to -\infty,$$
    another contradiction. Thus $f\p *=0$, that is to say, $f(x(t))\to 0$. As above, let $\ell_k = \min \{f(x)/2:1/k\leq |x-\overline{x}|\leq \epsilon\}>0$ for all $k\in \N\p *$. Hence, for all $k\in \N\p *$ and sufficiently large $t >0$, we have $x(t) \in [f\leq \ell_k]\cap B_{\epsilon}(\overline{x}) \subseteq B_{1/k}(\overline{x})$. We conclude that $x(t)\to \overline{x}$.
\end{proof}

\begin{proof}[Proof of \cref{fact:bounded}]
    Let $U_x$ be a neighborhood of $x\in \overline{X}$ such that $F(U_x)$ is bounded. By compactness, the open cover $\{U_x:x\in \overline{X}\}$ admits a finite subcover $U_{x_1},\hdots,U_{x_p}$. Thus $F(X)\subseteq F(\overline{X})\subseteq F(U_{x_1})\cup\cdots\cup F(U_{x_p})$ is bounded. 
\end{proof}

\begin{proof}[Proof of \cref{fact:usc_bounded}]
    ($\Longrightarrow$) Let $\overline{x}\in \R\p n$. Since $B_1(F(\overline{x}))$ is a neighborhood of the compact set $F(\overline{x})$, by upper semicontinuity there exists a neighborhood $U$ of $\overline{x}$ such that $F(U)\subseteq V$. In particular, $F(U)$ is bounded. 
    Suppose $\gph F \ni (x_k,y_k)\to(\overline{x},\overline{y})\notin \gph F$. As $F(\overline{x})\p c$ is open, there exists $r>0$ such that $\overline{B}_r(\overline{y}) \subseteq F(\overline{x})\p c$. Since $V = [\overline{B}_r(\overline{y})]\p c$ is open and contains $F(\overline{x})$, by upper semicontinuity there exists a neighborhood $U$ of $\overline{x}$ in $\R\p n$ such that $F(U) \subseteq V$. Since $x_k \to \overline{x}$, eventually $x_k \in U$, and so $y_k \in F(x_k) \subseteq V$, i.e., $y_k \notin B_r(\overline{y})$, a contradiction. 

    ($\Longleftarrow$) Let $\overline{x} \in \R\p n$, $F(\overline{x}) = \gph F \cap (\{\overline{x}\}\times \R\p m)$ is closed as $\gph F$ is closed. Since $F(\overline{x})$ is also bounded, it is compact. Let $V$ be a neighborhood of $F(\overline{x})$ in $\R\p m$. Assume that $F(U) \cap (\R\p m \setminus V) \neq \emptyset$ for any neighborhood of $\overline{x}$ in $\R\p n$. Then there exists a sequence $(x_k,y_k)\in \gph F$ such that $x_k \to \overline{x}$ and $y_k \notin V$. Since $F$ is locally bounded, $y_k$ has a limit point $\overline{y}\notin V$. But $\gph F$ is closed, so $(\overline{x},\overline{y}) \in \gph F$, i.e., $\overline{y}\in F(\overline{x})\subseteq V$, a contradiction.
\end{proof}

\begin{proof}[Proof of \cref{fact:equilibrium}]
    Suppose $v = P_{F(\overline{x})}(0)\neq 0$. Then $\langle 0-v,w-v\rangle \leq 0$ for all $w\in F(\overline{x})$, namely, $\langle v,w\rangle \geq |v|\p 2>0$. Since $\gph F$ is closed by \cref{fact:usc_bounded}, there exists $\epsilon>0$ such that $\langle w,v\rangle\geq |v|\p 2/2$ for all $x\in B_\epsilon(\overline{x})$ and $w\in F(x)$. Then for any maximal solution $x:[0,T)\to\R\p n$ to $\dot{x}\in F(x)$ such that $x(t)\in B_\epsilon(\overline{x})$ for all $t\in [0,T)$, we have $T=\infty$ and $\langle x'(t),v\rangle\geq |v|\p 2/2$ for almost every $t\in (0,\infty)$. Thus $\langle x(t),v\rangle \geq |v|\p 2t/2 \to \infty$, a contradiction.
\end{proof}

\begin{proof}[Proof of \cref{fact:neighborhood}]
    $U$ is a neighborhood of each point $x\in X$, so there exists $r_x>0$ such that $B_{r_x}(x)\subseteq U$. The open cover $\{B_{r_x}(x):x\in X\}$ of $X$ admits a finite subcover since $X$ is compact, yielding $r>0$ such that $B_r(X)\subseteq U$.
\end{proof}

\begin{proof}[Proof of \cref{fact:co_usc}]
    By Carathéodory’s theorem, $\co F$ has compact convex values. Let $V$ be a neighborhood of $\co F(\overline{x})$ for some $\overline{x}\in\R\p n$. By \cref{fact:neighborhood}, there exists $r>0$ such that $W = B_r(\co F(\overline{x})) \subseteq V$. Since $W$ is a neighborhood of $F(\overline{x})$, there exists a neighborhood $U$ of $\overline{x}$ such that $F(U)\subseteq W$. As $W = \co F(\overline{x})+B_r(0)$ is convex, $\co F(U)\subseteq W \subseteq V$.
\end{proof}

\begin{proof}[Proof of \cref{lemma:tracking}] 
    Let $\{\alpha\p i\}_{i \in \mathbb{N}} \subseteq (0,\min\{\widehat\alpha,T/2\})$ be such that $\alpha\p i \to 0$. To each $i\in \N$, we attribute sequences $\{\alpha_k\p i\}_{k\in \N} \subseteq (0,\alpha\p i]$ and $\{x_k^i\}_{k\in\mathbb{N}}\subseteq\R\p n$ such that $x_0\p i \in X$ and $x_{k+1}\p i \in x_k\p i + \alpha_k\p i F(x_k\p i)$ for all $k\in \N$. Also, let $\{t_k\p i\}_{k\in\N}$ be such that 
    $$t_0\p i = 0 ~~~\text{and}~~~ \forall k\in\N\p *, ~ t_k\p i = \alpha_0\p i + \cdots + \alpha_{k-1}\p i.$$  
    By assumption, if $t_k\p i \leq T$, then $x_k\p i \in B_r(X)$. Consider the linear interpolation $y\p i:[0,T]\rightarrow\mathbb{R}^n$ defined by $$y\p i(t) = x_k\p i + (t-t_k\p i)(x_{k+1}\p i-x_k\p i)/\alpha_k\p i$$ for all $t\in [t_k\p i , \min\{t_{k+1}\p i,T\}]$ and $k\in \N$ such that $t_k\p i \leq T$. Since $F$ is upper semicontinuous with compact values, by \cref{fact:bounded} and \cref{fact:usc_bounded} we have $L = \sup \{ |u| : u\in F(B_r(X)) \}<\infty$. Naturally, $d(y\p i(t),\co X)\leq r+\alpha\p iL\leq r+TL/2$ for all $t\in [0,T]$, where the second term accounts for the last iterate, possibly outside $B_r(X)$. Observe that $(y\p i)'(t) = (x_{k+1}^i-x_k^i)/\alpha_k\p i \in F(x_k^i)$ for all $t\in (t_k\p i , \min\{t_{k+1}\p i,T\})$ and $k\in\N$ such that $t_k\p i \leq T$. Hence $|(y\p i)'(t)| \leq L$ for almost every $t\in (0,T)$. By successively applying the Arzel\`a-Ascoli and the Banach-Alaoglu theorems (see \cite[Theorem 4 p. 13]{aubin1984differential}), there exists a subsequence (again denoted $\{\alpha\p i\}_{i\in \mathbb{N}}$) and an absolutely continuous function $x:[0,T]\rightarrow \mathbb{R}^n$ such that $y\p i(\cdot)$ converges uniformly to $x(\cdot)$ and $(y\p i)'(\cdot)$ converges weakly to $x'(\cdot)$ in $L^1([0,T],\mathbb{R}^n)$. Furthermore, for all $t\in (t_k\p i , \min\{t_{k+1}\p i,T\})$ and $k \in \N$ such that $t_k\p i \leq T$,
\begin{align*}
    (y\p i(t),(y\p i)'(t)) 
    & = \left(x_k^i + (t-t_k\p i)\frac{x_{k+1}^i-x_k^i}{\alpha_k\p i},\frac{x_{k+1}^i-x_k^i}{\alpha_k\p i}\right) \\
    & \in \left(\{x_k^i\} + (t-t_k\p i)F(x_k^i)\right) \times F(x_k^i)  \\
    & = \{x_k^i\} \times F(x_k^i) + (t-t_k\p i)F(x_k^i) \times \{0\}  \\
    & \subseteq \gph F + B_{L\alpha\p i}(0)\times \{0\}.
\end{align*}
Since $F$ is upper semicontinuous with compact values, $\co F$ is upper semicontinuous with closed convex values by \cref{fact:co_usc}. According to \cite[Convergence Theorem p. 60]{aubin1984differential}, it follows that $x'(t) \in \co F(x(t))$ for almost every $t\in (0,T)$. The sequence of initial points $\{x^i(0)\}_{i\in\mathbb{N}}$ lies in $X$, hence its limit $x(0)$ lies in $\overline{X}$. As a result, $x(\cdot)$ is a solution to the differential inclusion \eqref{subeq:DI}. 

To sum up, for every positive sequence $\{\alpha\p i\}_{i \in \mathbb{N}}$ converging to zero, there exists a subsequence for which the corresponding linear interpolations uniformly converge towards a solution of the differential inclusion \eqref{subeq:DI}. We now argue by contradiction and assume that there exists $\eta>0$ such that for all $\overline{\alpha}>0$, there exist $\{\alpha_k\}_{k\in\N} \subseteq (0,\overline{\alpha}]$ and $\{x_k\}_{k \in \mathbb{N}}\subseteq \R\p n$ satisfying \eqref{eq:Euler} such that, for any solution $x(\cdot)$ to the differential inclusion \eqref{subeq:DI}, it holds that $|y(t) - x(t) | > \eta$ for some $t \in [0,T]$. We can then generate a positive sequence $\{\alpha\p i\}_{i \in \mathbb{N}}$ converging to zero such that, for any solution $x(\cdot)$ to the differential inclusion \eqref{subeq:DI}, it holds that $|y\p i(t) - x(t)| > \eta$ for some $t \in [0,T]$. Since there exists a subsequence $\{\alpha\p{\varphi(i)}\}_{i \in \mathbb{N}}$ such that $y^{\varphi(i)}$ uniformly converges to a solution to the differential inclusion \eqref{subeq:DI}, we obtain a contradiction.
\end{proof}

\section*{Acknowledgments}
I wish to thank Wenqing Ouyang for his valuable feedback.

\bibliographystyle{abbrv}    
\bibliography{references}
\end{document}